\documentclass[onefignum,onetabnum]{siamart171218}

\usepackage{amsfonts}
\usepackage{graphicx}
\usepackage{epstopdf}
\usepackage{algpseudocode}
\usepackage{subcaption} 

\usepackage{amssymb}
\usepackage{bm}
\ifpdf
  \DeclareGraphicsExtensions{.eps,.pdf,.png,.jpg}
\else
  \DeclareGraphicsExtensions{.eps}
\fi

\newsiamremark{remark}{Remark}
\newsiamremark{hypothesis}{Hypothesis}
\crefname{hypothesis}{Hypothesis}{Hypotheses}
\newsiamthm{claim}{Claim}

\headers{Equilibrated Averaging Residual Method}{Cuiyu He}

\title{Equilibrated Averaging Residual Method: A General Approach to Conservative Flux Recovery}

\author{Cuiyu He\thanks{Department of Mathematics, University of Georgia, Athens, GA
  (\email{cuiyu.he@uga.edu}).}}

\usepackage{amsopn}

\renewcommand{\theequation}{\thesection.\arabic{equation}}
\def\@eqnnum{{\reset@font\rm (\theequation)}}

\def\abstract{
\advance \rightskip by 10mm
\advance \leftskip by 10mm
\vspace{-0.8em}
\noindent
\small{\bf Abstract.}
}

\def\XXint#1#2#3{{\setbox0=\hbox{$#1{#2#3}{\int}$}
\vcenter{\hbox{$#2#3$}}\kern-.5\wd0}}

\def\a{\alpha}

\renewcommand\o{\omega}\renewcommand\O{\Omega}
\def\S{\Sigma}

\newcommand{\bsigma}{\mbox{\boldmath$\sigma$}}
\newcommand{\btau}{\mbox{\boldmath$\tau$}}

\newcommand{\bftau}{\boldsymbol {\tau}}

\def\be{{\bf e}}

\def\bn{{\bf n}}

\def\bs{{\bf s}}

\newcommand{\bx}{\mbox{\boldmath$x$}}

\def\cE{{\cal E}}

\def\cT{{\cal T}}

\def\sD{{_D}}

\def\f12{\frac12}
\def\dfrac{\displaystyle\frac}

\def\p{\partial}

\newcommand{\gradt}{\nabla\cdot}

\def\divvr{{\rm div}}

\def\curll{{\rm curl}}

\newcommand{\tri}{|\!|\!|}

\newcommand{\bdm}{\begin{displaymath}}
\newcommand{\edm}{\end{displaymath}}
\newcommand{\beq}{\begin{equation}}
\newcommand{\eeq}{\end{equation}}
\newcommand{\beqa}{\begin{eqnarray}}
\newcommand{\eeqa}{\end{eqnarray}}
\newcommand{\beqas}{\begin{eqnarray*}}
\newcommand{\eeqas}{\end{eqnarray*}}

\ifpdf
\hypersetup{
  pdftitle={Equilibrated Averaging Residual Method: A General Approach to Conservative Flux Recovery},
  pdfauthor={C. He}
}
\fi

\def\cN{{\cal N}}
\newcommand{\jump}[1]{[\![ #1]\!]}

\begin{document}

\maketitle

\begin{abstract}
Many equilibrated flux recovery methods for finite element solutions rely on ad hoc or method-specific techniques, limiting their generalizability and efficiency. In this work, we introduce the Equilibrated Averaging Residual Method (EARM), a unified framework for flux recovery that not only reproduces state-of-the-art locally conservative fluxes but also enables the derivation of new equilibrated fluxes with improved properties. In this paper, EARM is applied to conforming, nonconforming, and discontinuous Galerkin methods, ensuring local conservation and robust a posteriori error estimation. 
Despite the unified nature of the variational problem, the framework retains the flexibility to fully leverage the inherent properties of finite element spaces.
Moreover, EARM offers explicit and computationally efficient flux reconstructions for all methods in two dimensions. In three dimensions, only simple local problems need to be solved for the conforming finite element methods.

\end{abstract}

\begin{keywords}
 Adaptive Finite Element Method; equilibrated flux recovery; a posteriori error estimation.
\end{keywords}

\begin{AMS}
65N30 65N50 
\end{AMS}


\section{Introduction}\label{intro}


An accurate and locally conservative recovered flux  (or equilibrated flux) for solutions of finite element method (FEM) plays an essential role in many applications, including a posteriori error estimation \cite{Ai:07b,braess2008equilibrated,Ve:09,BFH:14,ErnVo2015,CaCaZh:20,cai2021generalized} (a key component in adaptive mesh refinement procedure), compatible transport in heterogeneous media \cite{odsaeter2017postprocessing}, velocity reconstruction in porous media  \cite{ern2009accurate,vohralik2013posteriori,bastian2014fully,capatina2016nitsche}, to name a few. In this paper, we will focus on the application of conservative flux recovery in the a posteriori error estimation.  Equilibrated a posteriori error estimation has attracted much interest due to the guaranteed reliability bound of the conforming error, meaning the reliability constant is equal to one. This property implies that they are perfect for discretization error control on both coarse and fine meshes. Error control on coarse meshes is important but difficult for computationally challenging problems.

 The mathematical foundation of equilibrated a posteriori error estimator for the conforming finite element approximation is the Prager-Synge identity \cite{prager1947approximations} that is valid in $H^1(\Omega)$. 
Prager-Synge identity for piecewise $H^1(\Omega)$ (discontinuous) finite element approximations in both two and three dimensions is generalized in \cite{cai2021generalized}. It is shown that the error can be split into conforming and nonconforming errors, where the conforming error is guaranteed to be bound by the equilibrated error estimator, which is the energy norm of the difference between the numerical  and a recovered equilibrated flux. 

Many equilibrated flux recovery methods for finite element solutions rely on ad hoc or method-specific approaches. In this paper, we introduce a unified framework that recovers state-of-the-art locally conservative fluxes for various FEMs. In this paper, it is applied to conforming, nonconforming, and discontinuous Galerkin methods.
Despite the unified nature of the variational problem, the framework retains the flexibility to fully leverage the inherent properties of finite element spaces. 
 Since our method is derived from the residual operator based on the averaging flux, 
 we refer to this approach as the Equilibrated Averaging Residual Method (EARM) to distinguish it from the classical equilibrated residual method introduced in \cite{ainsworth2000posteriori}.

Discontinuous Galerkin (DG) methods are well known for inherently ensuring local equilibria. We refer to \cite{Ai:07b, ern2007accurate} and references therein for earlier works on explicit conservative flux reconstructions of linear and higher-order DG elements in elliptic problems. Notably, an obvious solution to the EARM for DG solutions naturally coincides with the results in \cite{Ai:07b} and \cite{ern2007accurate} for linear and arbitrary-order DG finite element solutions.

It is important to note that there are infinitely many solutions (locally conservative fluxes) to the variational problem introduced by the EARM, as its null space corresponds to the divergence-free flux  space. Our objective is to identify a solution that is both equilibrated and sufficiently accurate to ensure local efficiency in the a posteriori error analysis. Ideally, this solution can be constructed using only local data, avoiding the need to solve a global problem or local problems.

The nonconforming (NC) FEMs were first introduced to overcome the locking phenomenon for linear elasticity problem \cite{lee2003locking, brenner1992linear}. 
For NC methods applied to elliptic diffusion problems, explicit equilibrated flux recoveries in Raviert Thomas (RT) spaces have been constructed for the linear Crouzeix-Raviart (CR) in \cite{Ma:85}, (see \cite{ainsworth2005robust} in the context of the estimator), for the quadratic Fortin-Soulie elements in \cite{kim2012flux}, and for arbitrary odd-order NC methods in \cite{becker2016local, cai2021generalized}. These explicit methods are delicately designed and fully take advantage of the properties of the basis functions.
In this paper, we demonstrate that applying EARM naturally results in the explicit flux recoveries derived in \cite{Ma:85, cai2021generalized}, while also yielding novel recoveries.

Many researchers have also studied various methods of conservative flux recovery for conforming FEM
(e.g.,  \cite{ladeveze1983error, demkowicz1987adaptive, oden1989toward,  destuynder1999explicit, ainsworth2000posteriori, AiOd:93,larson2004conservative, vejchodsky2006guaranteed, braess2007finite, braess2008equilibrated, BrPiSc:09, Ve:09, CaZh:11,becker2016local, odsaeter2017postprocessing, CaCaZh:20, ErnVo2015}).
 Partition of unity (POU) is a commonly used tool for localization due to the lack of explicit construction technique. 
By using the POU, Ladev\`eze and Leguillon \cite{ladeveze1983error} initiated a local procedure 
to reduce the construction of an equilibrated flux to vertex patch-based local calculations. 
For the continuous linear finite element approximation to the Poisson equation in two dimensions, an equilibrated flux in the lowest order Raviart-Thomas space was explicitly constructed in \cite{braess2007finite, braess2008equilibrated}.
Without introducing a constraint minimization (see \cite{CaZh:11}), this explicit approach does not lead to a robust a posteriori error estimator concerning the coefficient jump.  The constraint minimization problem on each vertex 
patch can be solved by first computing an equilibrated flux and then calculating a divergence-free correction.  For recent developments relating to this method, see \cite{CaCaZh:20} and references therein. In \cite{ErnVo2015}, a unified method also based on the POU was developed. This method requires solving local mixed problems on a vertex patch for each vertex. In \cite{larson2004conservative, odsaeter2017postprocessing} a global problem is solved on the enriched piecewise constant DG space
to obtain the conservative flux.

In principle, the POU method can be uniformly applied to various finite element methods (FEMs), as it does not rely on any specific structure of the finite element basis functions. However, it is primarily used with the continuous Galerkin (CG) method, since simple recovery techniques are available for discontinuous Galerkin (DG) and nonconforming methods of odd orders.  
For the same reason, these methods can also be extended to nonconforming FEMs of arbitrary even order. Unsurprisingly, this approach tends to be relatively complex, as it requires solving star-patched local problems, which are either constrained \cite{CaZh:11} or formulated in a mixed setting \cite{ErnVo2015}.

For approaches beyond the POU method, we refer to \cite{becker2016local, ainsworth2000posteriori}, where two-dimensional Poisson problems are studied.  
In \cite{becker2016local}, the method is formulated as a unified mixed problem applicable to continuous, nonconforming, and discontinuous Galerkin (DG) methods in two dimensions. 
This approach enables the local construction of conservative fluxes for various FEMs.  However, extending this method to three dimensions for the conforming case is not straightforward, as the constraints imposed on the skeleton space in 2D do not have an obvious generalization to 3D.

The equilibrated residual method presented in \cite{ainsworth2000posteriori} (Chapter 6.4) also constructs a conservative flux for the conforming FEM without the POU. Instead of directly solving for the flux, this approach first determines its moments, enabling a natural localization to star-patch regions.
However, the local star-patch problem associated with each interior vertex and some boundary vertices has a one-dimensional kernel. To address this, an additional constraint is introduced, necessitating the use of the Lagrangian method to solve the local patch problem. Finally, a global assembly of the flux is performed based on the computed moments.

Applying the EARM directly to CG methods results in local patch problems similar to those in \cite{ainsworth2000posteriori} (see (5.5)).
To further enhance the equilibrated residual method, we introduce two approaches based on the averaging equilibrated residual method that eliminate the need for solving a constrained minimization problem using the Lagrangian method. Unlike \cite{becker2016local, ainsworth2000posteriori}, our method directly solves for the flux in the RT space, thereby avoiding the need for a final assembly based on moments. Additionally, we note that both methods extend naturally to NC methods of even orders, for which existing simple recovery techniques do not currently exist.

In the first method, we restrict the correction flux to the orthogonal complement of the divergence-free space. We prove that the space formed by functions defined as the jump of any function in the DG space across facets is a subspace of this orthogonal complement. This approach enables the formulation of a variational problem that guarantees a unique solution.  
We note that, in the case of the lowest-order ($0-th$) DG space, our method coincides with those in \cite{larson2004conservative, odsaeter2017postprocessing}. 

We note that the variational problem arising from the first method leads to a global problem. However, for linear and higher-order elements, it can be localized using the approach in \cite{becker2016local} by replacing the exact integral on facets with an inexact Gauss-Lobatto quadrature.
With Gauss-Lobatto quadrature in two dimensions, our recovered flux coincides with that in \cite{becker2016local}. Furthermore, the recovered flux can be computed entirely explicitly, as shown in \cite{capatina2024robust}. In three dimensions, our method requires solving only minor local problems that are neither constrained nor mixed star patch problems.

In the second method, we propose a localization approach that employs a partial POU. Specifically, we apply the POU only to the linear residual operator while leaving the bilinear form unchanged. Our method can be seen as a variation of the classical equilibrated residual method, where the residual operator \( r(v) \) is replaced by \( r(v \lambda_{z}) \), with \( \lambda_{z} \) being the linear barycentric basis function associated with the vertex \( z \).  
The advantage of our method is that it directly solves for the flux rather than its moments. Similar to \cite{ainsworth2000posteriori}, our local problems also admit a one-dimensional null space. However, since we solve for the flux directly, we can efficiently handle the constrained minimization problem by first computing a specific solution and a null-space solution of the unconstrained linear system, then forming an appropriate linear combination to satisfy the imposed constraint.

Since EARM systematically reproduces known equilibrated fluxes, its numerical performance aligns with well-established results in the literature. Given space constraints, we do not include additional numerical experiments. However, the new fluxes introduced by EARM are expected to exhibit the same theoretical guarantees and efficiency properties demonstrated in our analysis.



The remainder of this paper is organized as follows.
In \cref{sec:2}, we introduce the model problems and necessary notation. Next, in \cref{sec:3}, we demonstrate how our method systematically recovers existing state-of-the-art results for discontinuous Galerkin methods. In \cref{sec:4}, we derive both established recovered fluxes and novel recoveries for nonconforming methods. 

In \cref{sec:5}, we present the recovery techniques for conforming finite element methods (FEM), first deriving the existing equilibrated residual method and then introducing two enhanced approaches.
The automatic reliability results are provided in \cref{sec:6}. 
Finally, in \cref{sec:7}, we prove the robust efficiency of all newly developed fluxes introduced in this paper, demonstrating that the efficiency constant remains independent of the jump in the coefficients.

\section{Model problem}\label{sec:2}
\setcounter{equation}{0}

Let $\O$ be a bounded polygonal domain in $\mathbb{R}^d, d=2,3$, with Lipschitz 
boundary $\partial \O = \overline\Gamma_D \cup \overline \Gamma_N$, where
$ \Gamma_D \cap \Gamma_N = \emptyset$.
For simplicity, assume that
$\mbox{meas}_{d-1}(\Gamma_D) \neq 0$.
Considering the diffusion problem:
\begin{equation}\label{pde}
	-\gradt (A \nabla u)  =  f   \quad\mbox{in} \quad  \O,
\end{equation} 
with boundary conditions 
\[
	u = 0 \; \mbox{ on }  \Gamma_D \quad \mbox{and} \quad
	-A \nabla u \cdot \bn=g \;\mbox{ on } 
	\Gamma_N,
\]
where $\nabla \cdot$ and $\nabla$ are the respective divergence and gradient operators; $\bn$ is the outward unit 
vector normal to the boundary; $f \in L^2(\O)$ and $g\in H^{-1/2}(\Gamma_N)$ are given scalar-valued functions; and the diffusion coefficient $A(x)$ is symmetric, positive definite, piecewise constant full tensor. 


 In this paper, we use the standard notations and definitions for the Sobolev spaces. Let
\[
	H_D^1(\O) =\left\{v \in H^1(\O) \,:\,
	v=0 \mbox{ on } \Gamma_D \right\}.
\]
Then the corresponding variational problem of (\ref{pde}) is to  find $u \in H^1_D(\O)$ such that 
\begin{equation} \label{vp}
	a(u,\,v):= (A\nabla u, \nabla v) = (f, v)_{\O}- \left<g, v\right>_{\Gamma_N},
	\quad \forall  \;v\in H_D^1(\O),
\end{equation}
where $(\cdot, \cdot)_{\omega}$ is the $L^2$ inner product on the domain $\o$. 
The subscript $\omega$ is omitted from here to thereafter when $\o=\O$. 

\subsection{Notations}
 Let $\cT_h=\{K\}$ be a finite element partition of $\O$ that is regular, and denote 
 by $h_K$ the diameter of the element $K$. 
 Denote the set of all facets of the triangulation $\cT_h$ by
  \[
 	\cE := \cE_I \cup \cE_D \cup \cE_N,
 \]
 where $\cE_I$ is the set of interior element facets, and $\cE_D$ and $\cE_N$ are the sets of 
boundary facets belonging to the respective $\Gamma_D$ and $\Gamma_N$.
  For each $F \in \cE$, denote by $h_F$ the length of $F$ and by
 $\bn_F$ a unit vector normal to $F$.
 Let $K_F^+$ and $K_F^-$ be the two elements sharing the common facet $F \in \cE_I$ 
 such that the unit outward normal of $K_F^-$ coincides with $\bn_F$. When $F \in \cE_D \cup \cE_N $,
 $\bn_F$ is the unit outward normal to $\partial \O$ and denote by $K_F^-$ the element having the facet $F$.
 Note here that the term facet refers to the $d-1$ dimensional entity of the mesh. In 2D, a facet is equivalent to an edge, and in 3D, it is equivalent to a face. We also denote by $\cN$ the set of all vertices and by $\cN_{I}\subset \cN$ the set of interior vertices.
 
For each $K \in \cT_h (F \in \cE)$, let $\mathbb{P}_k(K) (\mathbb{P}_k(F))$ denote the space of polynomials of degree at most $k$ on $K (F)$. 
For each $K \in \cT_{h}$, we define a sign function $\mbox{sign}_{K}(F)$ on $\cE_{K} := \{F:  F \in \cE,
\mbox{ and } F \subset \partial K\}$:
\[
\mbox{sign}_{K}(F) = 
\begin{cases}
1 & \mbox{ if } \bn_{F} = \bn_{K}|_{F},\\
-1 & \mbox{ if } \bn_{F} =-\bn_{K}|_{F}.
\end{cases}
\]

\section{Applying the EARM for DG FEM}\label{sec:3}
In this subsection, we introduce and apply the EARM for DG FEM. This method fully explores the features of the DG space and explicitly constructs locally equilibrated fluxes that coincide with the existing state-of-the-art results. 

Define
\[
DG(\cT_h,k) = \{v\in L^2(\O)\,:\,v|_K\in \mathbb{P}_k(K),\quad\forall\,\,K\in\cT_h\}.
\]
Denote the $H(\mbox{div}; \O)$ conforming Raviart-Thomas (RT)  space of index $k$ with respect to $\cT_h$
by
{{\[
RT(\cT_h,k) = \left\{ \btau \in H(\mbox{div};\O) \,: \, \btau|_K \in RT(K,k),  \;\forall\, K\in\cT_h \right\},
\]}}
where $RT(K,k) = \mathbb{P}_{k}(K)^d + \bx \, \mathbb{P}_{k}(K) $.
Also let
{{
\[
RT_f(\cT_h,k) = \left\{ \btau \in RT(\cT_h,k): \gradt \btau =f_{k} \,\mbox{in} \, \O \mbox{ and} \, \btau \cdot \bn_F=
  g_{k,F} \, \mbox{on} \,F \in \cE_N \right\},
\]}}
where $f_k$ is the $L^2$ projection of $f$ onto $DG(\cT_h,k)$ and $g_{k,F}$ is the $L^2$ projection of $g|_F$ onto $\mathbb{P}_k(F)$.

  For each $F \in \cE_I$, we define the following weights:
  $\omega_F^\pm = \dfrac{ \a_{F}^\mp}{\a_F^-+\a_F^+}$ 
where $\a_F^\pm = \lambda(A|_{K_F^\pm})$ and  $\lambda(M)$ is the maximal eigenvalue of the matrix $M$. We also define  the following weighted average and jump operators:
\[
	\{v\}_w^F = \begin{cases} w_F^+ v_F^+ + w_F^- v_F^-, & F \in \cE_I,\\
	v ,& F \in \cE_D \cup \cE_N,
	\end{cases}
	\{v\}_F^w = \begin{cases} w_F^- v_F^+ + w_F^+ v_F^-, & F \in \cE_I,\\
	0,& F \in \cE_D \cup \cE_N,
	\end{cases}
 \]
 \[
		\jump{v}|_F = \left\{
	\begin{array}{ll}
		v|_{F}^- - v|_{F}^+,&\forall \,F \in \cE_I,\\[2mm]
		v|_F^-, & \forall \, F \in \cE_D \cup \cE_N.
	\end{array}
	\right.
 \]
The weighted average defined above is necessary to ensure the robustness of the error estimation with respect to the jump of $A$, see \cite{cai2017residual}. 
It is easy to show that for any $F \in \cE_{I}$,
\begin{equation} \label{weight1}
w_F^{\pm}\a_F^{\pm} \le \sqrt{\a^{\pm}\a_{F,min}}, \quad\,\,
 \frac{\omega_F^+}{\sqrt{\alpha_F^-}} \le \sqrt{\frac{1}{\alpha_{F,max}}},
	\,\quad \mbox{and} \quad\,
	\frac{\omega_F^-}{\sqrt{\alpha_F^+}} \le \sqrt{\frac{1}{\alpha_{F,max}}},
\end{equation}
where $\a_{F,min} = \min(\a_{F}^{+}, \a_{F}^{-})$ and $\a_{F,max} = \max(\a_{F}^{+}, \a_{F}^{-})$.

We will also use the following commonly used identity:
 \begin{equation}\label{jump-id}
 \jump{ u v}_F = \{v\}^w_F\, \jump{u}_F + \{u\}_w^F\, \jump{v}_F, \quad \forall F \in \cE.
 \eeq

We introduce the following  DG formulation for (\ref{pde}): find $u\in V^{1+\epsilon}(\cT_h)$ with $\epsilon >0$ and $V^{s}(\cT_h) = \{v \in L^2(\O), v|_K \in H^s(K) \mbox{ and } \gradt A \nabla v \in L^2(K), \forall K \in \cT_h\}$
 such that
 \begin{equation}\label{DGV}
 a_{dg}(u,\,v) = (f,\,v)   -
\left< g, v \right>_{\Gamma_N}, \quad\forall\,\, v \in V^{1+\epsilon}(\cT_h), 
\end{equation}
where 
\begin{equation}
	\begin{split}
a_{dg}(u,v)&=(A\nabla_h u,\nabla_h v)
 +\sum_{F\in\cE \setminus \cE_N}\int_F\gamma  \dfrac{\a_{F,min}}{h_F} \jump{u}
 \jump{v}\,ds 
 -\sum_{F\in\cE  \setminus \cE_N}\int_F\{A\nabla
u\cdot\bn_F\}_{w}^F \jump{v}ds 
 \\
 &+ \delta
\sum_{F\in\cE  \setminus \cE_N}\int_F\{A\nabla
v\cdot\bn_F\}_{w}^F \jump{u}ds .
\end{split}
\end{equation}
Here,  $\nabla_h$ is the discrete gradient operator defined elementwisely, 
and $\gamma$ is a positive constant, $\delta$ is a constant that takes value $1, 0$ or $-1$. 

The DG solution is to
seek $u^{dg}_k \in DG(\cT_h,k)$ such that
{{\begin{equation}\label{problem_dg}
a_{dg}(u^{dg}_k,\, v) = (f,\,v)_\O-  \left< g, v \right>_{\Gamma_N}\quad \forall\, v\in DG(\cT_h,k).
\end{equation}}}

We now introduce the EARM.

Observe that the numerical , $-A \nabla u_k^{dg}$, does not necessarily belong to the space $H(\mbox{div}; \O)$. In the {first step}, we define a functional 
\[
\tilde \bsigma_s: u_h \in DG(\cT_h, k) \rightarrow  \tilde \bsigma_s(u_h) \in RT(\cT_h,s), \quad
0 \le s \le k,
\]
  such that
 {\begin{equation} \label{rt:1:a}
 \int_F \tilde \bsigma_{s}(u_h) \cdot \bn_F \phi \,ds = 
 \begin{cases}
 - \int_F\{A \nabla u_h \cdot \bn_F\}_w^F \phi \,ds& \forall F \in \cE \setminus \cE_N,\\
\int g \phi \,ds, & \forall F \in \cE_N,
\end{cases}
\quad \forall \phi \in \mathbb{P}_s(F).
 \end{equation}}
 and when $s\ge1$, additionly satisfies that
\begin{equation}\label{rt:1:dg}
 \begin{split}
 	&( \tilde\bsigma_{s}(u_h) , \bm{\psi})_K =-(A \nabla u_h , \bm{\psi} )_K, \quad\forall K \in \cT_h, \forall  \bm{\psi} \in \mathbb{P}_{s-1}(K)^d.
\end{split}
 \end{equation}
 Since the recovered flux $\tilde \bsigma_{s}(u_{h})$ uses the averaging  on the facets, we will refer to this  as the \textit{weighted averaging }.
 
  For simplicity, we assume that $g|_{F} = g_{k-1,F}$. Since $\{A \nabla u_h \cdot \bn_F\}_w^F  \in \mathbb{P}_{k-1}(F)$, it is easy to see that when $s \ge k-1$, there holds
  {\begin{equation} \label{rt:1:aa}
 \tilde \bsigma_{s}(u_h) \cdot \bn_F = 
 \begin{cases}
 - \{A \nabla u_h \cdot \bn_F\}_w^F& \forall F \in \cE \setminus \cE_N,\\
 g , & \forall F \in \cE_N.
\end{cases}
 \end{equation}}

 

It is easy to see that the weighted averaging  $\tilde\bsigma_s^{dg}:=\tilde \bsigma_{s}(u_{k}^{dg}) \in H(\mbox{div};\O)$, however, is not necessarily locally conservative.
In the second step, we aim to find a correction  $ \bsigma_s^{\Delta} \in RT(\cT_h, s)$ such that 
\begin{equation}\label{DG-final-}
\hat\bsigma_h^{dg}: = \tilde\bsigma_{s}^{dg}+  \bsigma_s^{\Delta} \in RT_f(\cT_h,s).
\end{equation}

Define the \textit{averaging residual operator} for any $u_{h}$ being a finite element solution:
\begin{equation}\label{r(v)}
r_{s}(v) = \sum_{K}r_{s,K}(v), \quad r_{s,K}(v) := (f - \gradt \tilde \bsigma_s(u_{h}),v)_K.
\end{equation}
Our goal is to find  $ \bsigma_s^{\Delta} \in RT(\cT_h, s)$ such that
\begin{equation}\label{-equation}
  (\gradt \bsigma_{s}^{\Delta},v) =  r_{s}(v) \quad \forall v \in DG(\cT_{h},s).
\end{equation}

 By integration by parts, the definition of $ \tilde \bsigma_s^{dg}$ and \cref{problem_dg}, for any $v \in \mathbb{P}_s(K), 0 \le s \le k$ there holds
\begin{equation}\label{2.8}
\begin{split}
&r_{s}(v) =  (f,v)_{K}  + ( \tilde \bsigma_s^{dg}, \nabla v)_{K} - < \tilde \bsigma_s^{dg} \cdot \bn_{K}, v>_{\partial K}\\
=&  (f,v)_K -(A \nabla_h u_k^{dg}, \nabla v)_K  +
\!\!\!\sum_{F\in\cE_K \setminus \cE_N}\!\!\!\int_F \{A\nabla u_k^{dg}\cdot\bn_F\}_{w}^F\jump{v}ds 
 -
\!\!\!\sum_{F\in\cE_K \cap \cE_N} \!\!\!\left< g,v\right>_F\\
=&
\sum_{F\in\cE_K \setminus \cE_N}\int_F\gamma  \dfrac{\a_{F,min}}{h_F} \jump{u_k^{dg}}
 \jump{v}\,ds
 + \delta
\sum_{F\in\cE_{K}  \setminus \cE_N}\int_F\{A\nabla
v\cdot\bn_F\}_{w}^F \jump{u_k^{dg}}ds.
\end{split}
\end{equation}
On the other side, from \cref{-equation} and integration by parts, we also have
{{\begin{equation}\label{2.9}
\begin{split}
r_{s}(v)= r_{s,K}(v) = (\gradt \bsigma_s^{\Delta},v)_K&= -( \bsigma_s^{\Delta}, \nabla v)_K
+\sum_{F\in\cE_K}\int_F \bsigma_s^{\Delta} \cdot \bn_F \jump{v}ds.
\end{split}
\end{equation}}}
Given \cref{2.8} with (\ref{2.9}), it is natural to define  $\bsigma_s^{\Delta} \in RT(\cT_{h},s)$ for all $K \in \cT_h$ such that
 \begin{equation} \label{rt:-cg}
\int_F \bsigma_s^{\Delta} \cdot \bn_F \phi \,ds = 
 \begin{cases}
\displaystyle  \gamma \dfrac{\a_{F,min}}{h_F}\int_F \jump{u_k^{dg}} \phi \,ds,& \forall F \in \cE \setminus \cE_N,\\
 0, & \forall F \in  \cE_N , \quad \forall \phi \in \mathbb{P}_k(F).
\end{cases}
 \end{equation}
 and, when, $s \ge 1$,
  \begin{equation}\label{rt::dg-cg}
 \begin{split}
 	&( \bsigma_s^{\Delta} , \bm{\psi})_K =-
	 \delta 
\sum_{F\in\cE_K  \setminus \cE_N}\int_F\{A\bm{\psi}\cdot\bn_F\}_{w}^F \jump{u_k^{dg}}ds.
	\quad \forall  \bm{\psi} \in \mathbb{P}_{s-1}(K)^d.
\end{split}
 \end{equation}

\begin{lemma}
The recovered flux $\hat \bsigma_{h}^{dg}$ defined in \cref{DG-final-} where $\bsigma_{s}^{\Delta }$ is defined in \cref{rt:-cg}-\cref{rt::dg-cg} belongs to $RT(\cT_{h},s)$ for some $ (0 \le s \le k)$. Futhermore, it is locally conservative, satisfying $ \hat \bsigma_s^{dg} \in RT_f(\cT_h,s)$. 
 \end{lemma}
 \begin{proof}
First $\hat \bsigma_{h}^{dg} \in RT(\cT_{h},s)$ is immediate.
To prove that $\hat \bsigma_h^{dg} \in RT_f(\cT_h,s)$, 
 by \cref{DG-final-}, integration by parts, the definition of  $ \bsigma_s^{\Delta}$,  \cref{2.8} and \cref{r(v)}, we have for any $v \in DG(K,s)$:
	\begin{equation}
	\begin{split}
		&( \nabla \cdot \hat\bsigma_h^{dg}, v)_{K} 
		= 
		(\nabla \cdot  \tilde\bsigma_s^{dg}, v)_{K} +
		(\nabla \cdot   \bsigma_s^{\Delta}, v)_{K} \\
		=& 
		(\nabla \cdot  \tilde\bsigma_s^{dg}, v)_{K} -
		( \bsigma_s^{\Delta},  \nabla v)_{K} +
		\sum_{F \in \cE_{K}} <  \bsigma_s^{\Delta} \cdot \bn_{F}, \jump{v}>_{F} \\
		=&
		(\nabla \cdot  \tilde\bsigma_s^{dg}, v)_{K} +
	\sum_{F\in\cE_K  \setminus \cE_N}\delta\int_F \{A\nabla v \cdot\bn_F\}_{w}^F \jump{u_{k}^{dg}}ds 
	+
		\sum_{F \in \cE_{K} \setminus \cE_{N}} \gamma \dfrac{\a_{F,min}}{h_F}\int_F  \jump{u_k^{dg}} \jump{v} \,ds  \\
		 =& (\nabla \cdot  \tilde\bsigma_s^{dg}, v)_{K} + r_{s}(v)  = 
		(\nabla \cdot  \tilde\bsigma_s^{dg}, v)_{K} + (f - \gradt \tilde \bsigma_s^{dg},v)_K
		= (f,v)_{K}.
		\end{split}
	\end{equation}
This completes the proof of the lemma.
 \end{proof}
The recovered flux here is similar to those introduced in \cite{Ai:07b,ern2007accurate} for $s =k$ but with modified weight. 

\begin{remark}
The solution to \cref{-equation} is not unique, as the space \( H(\text{curl},\O) \) lies in its null space, resulting in infinitely many possible solutions. Our focus is to identify solutions to \cref{-equation} that are computationally efficient—completely explicit, without requiring the solution of any local or global problems—and possess provable accuracy properties, such as robust local efficiency in a posteriori error estimation.

Indeed, thanks to the non-uniqueness, the variational problem provides flexibility to fully exploit the inherent properties of finite element solutions. By leveraging the complete discontinuity property of DG finite element solutions, the solution to \cref{-equation} coincides with existing state-of-the-art methods for various FEMs. 
\end{remark}

\begin{remark}
Since \( r_{s}(v) \) in the right-hand side of \cref{r(v)} is the residual based on the weighted averaging , we refer to our method as the \textit{Equilibrated Averaging Residual Method (EARM)} to distinguish it from the classical equilibrated residual method introduced in \cite{ainsworth2000posteriori}. 
\end{remark}

We now outline a pseudo-algorithm for the EARM.
Define
\[
RT_f(\cT_h,s',s) = \left\{ \btau \in RT(\cT_h,s'): \Pi_{s}(\gradt \btau) =f_{s} \,\mbox{in} \, \O \mbox{ and} \, \btau \cdot \bn_F=
  g_{s,F} \, \mbox{on} \,F \in \cE_N \right\}.
\]

%
%
%

\begin{algorithm}
\caption{The pseudocode for EARM}\label{algorithm}
\begin{algorithmic}[1]
\State \textbf{Input:} Finite element solution $u_{h}$, where $u_{h}$ is piecewise polynomial of order $k$.
\State Step 1: Compute the weighted averaging  by \cref{rt:1:a}-- \cref{rt:1:dg}:
$$\tilde{\boldsymbol{\sigma}}_{k-1}(u_{h}) \in RT(\mathcal{T}_{h},k-1)
\quad \mbox{or} \quad
\tilde{\boldsymbol{\sigma}}_{s}(u_{h}), 0 \le s \le k-1.
 $$ 
\State Step 2: Solve for a correction  flux: 
$$\boldsymbol{\sigma}_{s}^{\Delta} \in RT(\mathcal{T}_{h},s)$$  
for some $0 \le s \le k-1$ such that
\begin{equation}\label{-equation-general}
  (\gradt \bsigma_{s}^{\Delta},v)_K =  r_{s}(v) \quad \forall v \in  DG(\cT_{h},s).
\end{equation}
\State Step 3: Obtain the equilibrated flux for some $ 0 \le s \le k-1$:
\[
	\hat \bsigma_{h} = \tilde{\boldsymbol{\sigma}}_{k-1}(u_{h}) +  \bsigma_{s}^{\Delta} \in 
	RT_f(\cT_h, max(s,k-1),s), 
\]
or
\[
	\hat \bsigma_{h} = \tilde{\boldsymbol{\sigma}}_{s}(u_{h}) +  \bsigma_{s}^{\Delta} \in 
	RT_f(\cT_h, s,s). 
\]
\end{algorithmic}
\end{algorithm}



\section{Applying EARM for NC FEMs}\label{sec:4}
 We now turn to NC FEMs, where the lack of continuity introduces additional challenges in flux recovery.
We restrict our discussion on triangular meshes in two dimensions for the NC FEMs.   

Define
\begin{equation} 
NC(\cT_h,k)  \!=\! 
  \left\{\! v \in L^2(\O) \! :  v|_K \in \mathbb{P}_k(K) \forall \, K \!\in\! \cT_h,
	 \int_F \jump{v}\, p ds =0 \forall \, p \in \mathbb{P}_{k-1}(F),   F \in \cE_I  \!\right\}
\end{equation} 
and its subspace by
\[
	NC_{0,\Gamma_D}(\cT_h,k) =
	\left\{
	v \in 	NC(\cT_h,k) \! :\, \int_F v \, p \,ds =0, \; \forall \,p \in \mathbb{P}_{k-1}(F), F \in \cE_D
	\right\}.
\]
The NC finite element approximation of order $k (k \ge 1)$ is to find $u_k^{nc} \in NC_{0,\Gamma_D}(\cT_h,k)$ such that
\begin{equation}  \label{NC solution}
	 (A \nabla_h u_k^{nc}, \nabla_h v)
	=(f,v)_\O-\left<g, v\right>_{\Gamma_N}, \quad \forall \, v  \in NC_{0,\Gamma_D}(\cT_h,k).
\end{equation} 

Following \Cref{algorithm}, we firstly define the weighted averaging flux  $\tilde\bsigma_{k-1}^{nc} :=\tilde\bsigma_{k-1}(u_k^{nc})$ which is in  $RT(\cT_h, k-1)$ by \cref{rt:1:a}--\cref{rt:1:dg}.
The next step is to find $\bsigma_s^\Delta \in RT(\cT_h,s)$ such that  \cref{-equation-general} holds.
Unlike the DG scheme, there exists no trivial explicit expression for the correction  $ \bsigma_s^{\Delta}$ upon the NC scheme.

Different from the DG method, besides the local conservation equation \cref{-equation-general}, special attention is required for compatibility. More specifically,  the local flux recovery for NC solutions should also satisfy the following global compatibility equation, 
 
 \begin{equation}\label{eq:compatibility}
 r(v):= r_{k-1}(v) =0 \quad \forall v \in  NC_{0,\Gamma_D}(\cT_h,k).
 \end{equation}
 For simplicity, we will use $ r(v)$ for  $r_{k-1}(v)$ from here to thereafter.
Indeed, \cref{eq:compatibility} can be easily proved by the definition of $ \tilde \bsigma_{k-1}^{nc}$ and integration by parts as following:
 \begin{equation*}
 \begin{split}
 r(v)&= (f,v)_\O - (\gradt \tilde \bsigma_{k-1}^{nc},v)_{\O} = 
  (f,v)_{\O} +  ( \tilde \bsigma_{k-1}^{nc}, \nabla_{h} v)_{\O} -  \sum_{F \in \cE}( \tilde \bsigma_{k-1}^{nc} \cdot \bn_{F}, \jump{v}) 
  \\
  &=
        (f,v)_{\O}-(A \nabla_h u_k^{nc}, \nabla_h v)_{\O} -\left<g, v\right>_{\Gamma_N} =0.
  \end{split}
 \end{equation*}
Specifically, for any $v \in NC_{0,\Gamma_D}(\cT_h,k)$,  and $\mbox{supp}(v) = K$ for some $K \in \cT_{h}$, we also have $r(v) =0.$
The residual problem \cref{-equation-general} then should satisfy
\begin{equation}\label{compatbility-interior}
0 = r(v) =(\gradt \bsigma_{s}^{\Delta},v)_K
=
 -( \bsigma_s^{\Delta}, \nabla v)_K =0.
\end{equation}

From \cref{compatbility-interior}, it is natural to impose the following restriction for the interior degree of freedom of $\bsigma_s^{\Delta}$ for the NC solutions by
\[
( \bsigma_s^{\Delta},  \bm{\psi})_K =0 \quad \forall   \bm{\psi} \in \mathbb{P}_{s-1}(K)^d \quad \mbox{when } s\ge 1.
\]

Define the space for $s \ge 1$,
 \[
 \mathring{RT}(\cT_h,s) = \left\{ \btau\in RT(\cT_h,s) : ( \btau,  \bm{\psi})_K =0 \; \forall  \bm{\psi} \in \mathbb{P}_{s-1}(K)^d, K \in \cT_h, \btau \cdot \bn_F=
0,F \in \cE_{N} \right\}.
 \]
 In the case $s=0$, we let $ \mathring{RT}(\cT_h,0) =  {RT}(\cT_h,0)$.
 Then,  \cref{-equation-general}, becomes finding $ \bsigma_s^{\Delta} \in \mathring{RT}(\cT_{h},s)$ such that
 \begin{equation}\label{nc-compatbility}
\begin{split}
\sum_{F\in\cE \setminus \cE_{N}}\int_F \bsigma_s^{\Delta} \cdot \bn_F \jump{v}ds =r(v)
\quad \forall v \in DG(\cT_{h},s).
\end{split}
\end{equation}

We first provide the following lemma that provides compatibility.
\begin{lemma}\label{lem:compatibility}
Let $\mathcal{A}(\bftau_h,v) := \underset{F \in \cE \setminus \cE_{N}}{ \sum } \left<\bftau_h \cdot \bn_F, \jump{v} \right>_F$ be a bilinear form defined on $\bftau_h \in \mathring{RT}(\cT_h, s) \times DG(\cT_{h},k), \,  s \le k-1$.
There holds
 \begin{equation}\label{compatibility}
\begin{split}
	&\mathcal{A}(\bftau_h,v)= r(v) = 0, \quad\forall v \in NC_{0,\Gamma_D}(\cT_h,k).
\end{split}
\end{equation}
\end{lemma}
\begin{proof}
The proof of the lemma directly follows from \cref{eq:compatibility} and the property of the NC space.
\end{proof}

\begin{lemma}\label{lem:existence}
Assume that for the NC finite element solution $u_{k}^{nc}$, there exists a solution  $\bsigma_s^{\Delta} \in 
\mathring{RT}( \cT_h, s), s\le k-1$ such that
\begin{equation}\label{nc-form}
\mathcal{A}(\bsigma_s^{\Delta} ,v) = r(v) \quad \forall  v \in  DG(\cT_h,s), 
\end{equation}
where $r(v) = (f - \nabla \cdot  \tilde\bsigma_{k-1}^{nc} ,v)$. Then there holds
\[
 \tilde\bsigma_{k-1}^{nc}+  \bsigma_s^{\Delta} \in 
RT_f(\cT_h,k-1,s), \quad
\mbox{and} \quad
 \tilde\bsigma_{s}^{nc}+  \bsigma_s^{\Delta} \in 
RT_f(\cT_h,s,s).
\]

\end{lemma}
\begin{proof}
	To prove that $ \tilde\bsigma_{k-1}^{nc}+  \bsigma_s^{\Delta} \in RT_f(\cT_h,k-1,s)$,
it is sufficient to prove the following local conservation for any $v \in \mathbb{P}_{s}(K) \subset DG(\cT_h,s)$. From integration by parts and \cref{nc-form}, we have
	\begin{equation}\label{conservation-nc}
	\begin{split}
		&(\nabla \cdot( \tilde\bsigma_{k-1}^{nc}+  \bsigma_s^{\Delta}), v)_{K} 
		= 
		(\nabla \cdot  \tilde\bsigma_{k-1}^{nc}, v)_{K} -
		( \bsigma_s^{\Delta},  \nabla v)_{K} +
		(  \bsigma_s^{\Delta} \cdot \bn_{F}, \jump{v})_{\partial K} \\
		&= 
		(\nabla \cdot  \tilde\bsigma_{k-1}^{nc}, v)_{K} +
		(  \bsigma_s^{\Delta} \cdot \bn_{F}, \jump{v})_{\partial K} \\
		 &=(\nabla \cdot  \tilde\bsigma_{k-1}^{nc}, v)_{K} + \mathcal{A}(\bsigma_s^{\Delta},v)
		=  (\nabla \cdot  \tilde\bsigma_{k-1}^{nc}, v)_{K} + r(v) = (f,v)_{K}.
		\end{split}
	\end{equation}
Note that $( \bsigma_s^{\Delta},  \nabla v)_{K} =0$ since  $ \bsigma_s^{\Delta} \in \mathring{RT}( \cT_h, s)$ and $\nabla v \in \mathbb{P}_{s-1}(K)^{d}$.

Finally, $\tilde\bsigma_{s}^{nc}+  \bsigma_s^{\Delta} \in 
RT_f(\cT_h,s,s)$ is direct consequence of \cref{conservation-nc} and the fact that
\[
\int_{K} \nabla \cdot  \tilde \bsigma_{s}^{nc} \cdot  p \,dx 
= \int_{K} \nabla \cdot \tilde \bsigma_{k-1}^{nc} \cdot p \,dx \quad \forall K \in \cT_{h}, p \in \mathbb{P}_{s}{(K)}.
\]

This completes the proof of the lemma.

\end{proof}
Our goal for the  recovery of NC solutions remains to seek explicit solutions to \cref{nc-form} by fully leveraging the properties of the finite element space.  Since the NC spaces for even and odd orders have dramatically different structures, we analyze them separately in two categories. 

\subsection{NC FEM of arbitrary odd order}
First, we describe basis functions of $NC(\cT_h,k)$, $k$ being \textit{odd} and their properties.
To this end, for each $K\in\cT_{h}$, let $m_k=\mbox{dim}(\mathbb{P}_{k-3}(K))$ for $k\ge 3$ and $m_k=0$ for $k< 3$. 
Denote by
$\{\bx_{j,K}, j=1, \cdots, m_k\}$
the set of all interior Lagrange points in $K$ with respect to
the space $\mathbb{P}_k(K)$ and
by $P_{j,K} \in \mathbb{P}_{k-3}(K)$ the nodal basis function corresponding to $\bx_j$, i.e., 
 \[
 P_{j,K}(\bx_{i,K}) =\delta_{ij} \,\,\mbox{ for } i=1,\,\cdots,\, m_k,
 \]  
where $\delta_{ij}$ is the Kronecker delta function.
For each edge $F \in \mathcal{E}$, let $\boldsymbol{s}_F$ and $\boldsymbol{e}_F$ be the two vertices of $F$ such that rotating the vector from $\boldsymbol{s}_F$ to $\boldsymbol{e}_F$ clockwise by 90 degrees aligns with $\boldsymbol{n}_{F}$.
For each $0 \le j \le k-1$, let
$L_{j,F}$ be the $j$th order Gauss-Legendre polynomial on $F$ such that $L_{j,F}(\be_F)=1$. 
Note that $L_{j,F}$ is an odd or even function when $j$ is odd or even on $F$. Hence, $L_{j,F}(\bs_F)=-1$ for odd $j$  and $L_{j,F}(\bs_F)=1$ for even $j$. 

For odd $k$, the set of degrees of freedom of $NC(\cT_h,k)$ (see Lemma 2.1 in \cite{ainsworth2008fully}) can be given by
\begin{equation}
	 \int_K v \,P_{j,K} \,dx, \quad j=1, \,\cdots, \,m_k
\end{equation} 
for all $K \in \cT_{h}$ and
\begin{equation}
	\int_F v \,L{_{j,F}} \,ds, \quad  j=0, \,\cdots,\, k-1
\end{equation}
for all $F \in \cE$.
For each $ K\in \cT_{h}$, and $i=1, \,\cdots, \, m_k$,
define the basis function $\phi_{i,K} \in NC(\cT_h,k)$ satisfying 
\begin{equation} \label{element-basis}
\left\{
	\begin{array}{lll}
	\int_{K'} \phi_{i,K} \,P_{j,K'} \,dx =  \delta_{ij}\delta_{KK'},  & \forall \,j=1,\, \cdots ,\,m_k, & \forall \, K' \in \cT_{h},\\[4mm]
	\int_{F} \phi_{i,K} \,L_{j,F} \,ds=0, & \forall \,j=0,\, \cdots,\, k-1, & \forall \,F \in \cE.
	\end{array}
	\right.
\end{equation} 
And for  each $ F \in \cE$, $i=0, \,\cdots, \, k-1$, define the basis function $\phi_{i,F} \in NC(\cT_h,k)$ satisfying
\begin{equation}\label{edge-basis}
\left\{
	\begin{array}{lll}
	\int_{K} \phi_{i,F} \,P_{j,K} \,dx=0, & \forall \,j=1, \,\cdots,\,m_k, & \forall \,K \in \cT_{h}, \\[4mm]
	\int_{F'} \phi_{i,F} \,L_{j,F'} \,ds = \delta_{ij} \delta_{_{FF'}},  & \forall \,j=0,\, \cdots,\, k-1, & \forall \,F' \in \cE.
	\end{array}
	\right.
\end{equation}
Then the NC finite element space is the space spanned by all these basis functions, i.e.,
\[
	NC(\cT_h,k) = \mbox{span}
	\left\{\phi_{i,K}:\, K \in \cT_{h}\right\}_{i=1}^{m_k} \oplus
	\mbox{span} \left\{\phi_{i,F}: \; F \in \cE \right\}_{i=0}^{k-1}.
\]

In what follows, we derive several locally conservative fluxes based on (\ref{nc-form}) for the NC solutions. 
\subsubsection{Case 1: $u_{h} \in NC(\cT_{h}, k = 2n+1), 
\hat \bsigma_{h} \in RT_f(\cT_h, k-1, k-1)$}

For NC space of any odd order $k$, (\ref{nc-form})  can be completely localized by testing with edge-based 
Gauss-Legendre basis functions denoted by $\{\phi_{i,F}^-, i=0, \cdots, k-1\}$ on each edge $F \in \cE \setminus \cE_{N}$.

Define  $\bsigma_{k-1}^{\Delta}  \in \mathring{RT}(\cT_h, k-1)$ such that
\begin{equation}\label{local__def}
\int_{F} \bsigma_{k-1}^{\Delta}  \cdot \bn_{F} L_{i,F} \,ds = { \| L_{i,F} \|_{F}^{2}} r(\phi_{i,F}^{-}),
	\quad \forall F\in \cE\setminus \cE_N, i = 0, \cdots, k-1.
\end{equation}

\begin{lemma}
	The $\bsigma_{k-1}^{\Delta} $ defined in \cref{local__def} satisfies that
\begin{equation}\label{compatbility}
\mathcal{A}(\bsigma_{k-1}^{\Delta}  ,v) = r(v) \quad \forall  v \subset DG(\cT_h,k-1). 
\end{equation}
Moreover,
we have
\begin{equation}\label{nc--recovery-a}
\hat\bsigma_h^{nc}: = \tilde\bsigma_{k-1}^{nc}+  \bsigma_{k-1}^{\Delta} \in RT_f(\cT_h, k-1, k-1), \quad 
\end{equation}
\end{lemma}

\begin{proof}
Choosing  $v_{i,F} = \phi_{i,F}^{-}$ in \cref{nc-form},
we have
\begin{equation}\label{localization1}
\begin{split}
\mathcal{A}(\bsigma_{k-1}^{\Delta},\phi_{i,F}^{-}) \!=\!\left<\bsigma_{k-1}^{\Delta} \cdot \bn_F, {\phi_{i,F}^{-}} \right>_{F}  \!\!=\!\! \dfrac{1}{\|L_{i,F}\|_{F}^{2}}\left<\bsigma_{k-1}^{\Delta} \cdot \bn_F, L_{i,F}\right>_{F}
\!=\!
 r(\phi_{i,F}^{-}),
	 \end{split}
\end{equation}
where we have used the property of $\phi_{i,F}$ that
\[
	\int_{F} \phi_{i,F}^{-} p_{s} \,ds = \dfrac{1}{\|L_{i,F}\|_{F}^{2}} \int_{F} L_{i,F}p_{s} \,ds 
	\quad \forall p_{s} \in \mathbb{P}_{s}(F), s \le k-1.
\]
Similarly, choosing  $v_{i,F} = \phi_{i,F}^{+}$ in \cref{nc-form} and using \cref{localization1} yiled
\begin{equation}\label{localization1+}
\begin{split}
\mathcal{A}(\bsigma_{k-1}^{\Delta},\phi_{i,F}^{+}) \!=\! \!-\left(\bsigma_{k-1}^{\Delta} \cdot \bn_F, {\phi_{i,F}^{+}} \right)_{F}  \!=\!
-\left(\bsigma_{k-1}^{\Delta} \cdot \bn_F, {\phi_{i,F}^{-}} \right)_{F} 
\!=\!-
 r(\phi_{i,F}^{-}) \!=\! r(\phi_{i,F}^{+}).
	 \end{split}
\end{equation}
To prove \cref{compatbility}, it is sufficient to prove for any $v \in \mathbb{P}_{k-1}(K)$,
	$K \in \cT_{h}$.
	Without loss of generality, assume that $K \in \cT_{h}$ is an interior element.
For each $v \in \mathbb{P}_{k-1}(K)$, there exist 
$a_{{j,F}} $  and 
$a_{{j,K}} $
 such that
\[
	v =\sum_{F \in \cE_K} \sum_{j=0}^{k-1} a_{{j,F}} \,\phi_{j,F}|_{K} 
	+ \sum_{j=1}^{m_k} a_{{j,K}} \,\phi_{j,K} := \sum_{F \in \cE_K} v_{F,K} + v_K \mbox{ in }K.
\] 
By a direct computation, we have $ \mathcal{A}(\bsigma_{k-1}^{\Delta} , v_{K}) =0$.
Combining with the fact that $r(v_{K})=0$ and \cref{localization1}, \cref{localization1+}, we have
\begin{equation}
\begin{split}
\mathcal{A}(\bsigma_{k-1}^{\Delta} ,v) &= 
\sum_{F \in \cE_{K}} \sum_{j=0}^{k-1} a_{{j,F}}  \mathcal{A}(\bsigma_{k-1}^{\Delta} , \phi_{j,F}|_{K})
\\
&=
\sum_{F \in \cE_{K}} \sum_{j=0}^{k-1} a_{{j,F}} r(\phi_{j,F}|_{K}) 
= \sum_{F \in \cE_{K}} r(v_{F,K}) = r(v).
\end{split}
\end{equation}
This completes the proof of \cref{compatbility}.
\cref{nc--recovery-a} is a direct application of  \cref{lem:existence}.
	This completes the proof of the lemma.
\end{proof}

The recovered flux here is the same as \cite{cai2021generalized} and \cite{becker2016local} for arbitrary odd orders NC FEM solutoons. For the first order Crouzeix-Raviart FEM, the recovered flux above is the same to the one retrieved in \cite{dari1996posteriori, ainsworth2005robust, cai2021generalized}.

\subsubsection{Case 2: $u_{h} \in NC(\cT_{h}, k = 2n+1), n\ge 0, 
\hat \bsigma_{h} \in RT_f(\cT_h, 0)$} In some applications, it is sufficient to require that the conservative flux lies in the lowest order RT space. In the following lemma, we provide an explicitly derived conservative flux  in
$RT_{f}(\cT_h, 0)$  for NC solutions of odd orders.
\begin{lemma}\label{lem:nc-case2}
Define  $\bsigma_{0}^{\Delta}  \in RT(\cT_h, 0)$ such that
\begin{equation}\label{local__def-A}
	\int_{F} \bsigma_{0}^{\Delta}  \cdot \bn_{F} \,ds = |F| r(\phi_{F}^{-}),
	\quad \forall F\in \cE \setminus \cE_{N}.
\end{equation}
where $ r(v) = (f - \nabla \cdot \tilde \bsigma_{k-1}^{nc}, v)$.
Then we have $\hat \bsigma_{h} :=  \bsigma_{0}^{\Delta} +  \tilde \bsigma_{k-1}^{nc}$ satisfies
\begin{equation}\label{nc-2}
\int_{K} \nabla \cdot \hat \bsigma_{h} \,dx = \int_{K} f \,dx  
\quad \forall K \in \cT_{h}
\end{equation}
and
\begin{equation}\label{local__def-Aa}
	 \bsigma_{0}^{\Delta} +  \tilde \bsigma_{0}^{nc} \in RT_{f}(\cT_{h}, 0).
\end{equation}
\end{lemma}
\begin{proof}
From \cref{lem:existence}, it is sufficient to prove
\begin{equation}\label{local-eq}
\mathcal{A}(\bsigma_0^{\Delta} ,v) = r(v) \quad \forall  v \in  DG(\cT_h,0).
\end{equation}
By the facts that
$L_{F} \equiv 1$, $\bsigma_0^{\Delta}  \cdot \bn_{F} $ is a constant on $F$, and $ \int_{F}{\phi_{F}} L_{F} \,ds =1$, we have  for any $F \in \cE_{K}$,
\begin{equation}\label{3.27}
\begin{split}
&\mathcal{A}(\bsigma_0^{\Delta} , \phi_{F}|_{K}) \!\!=\!\!
 \int_{F} \bsigma_0^{\Delta}  \cdot \bn_{F}
\jump{\phi_{F}|_{K}} \,ds = 
 \mbox{sign}_{K}(F) \bsigma_0^{\Delta}  \cdot \bn_{F} 
 \int_{F}{\phi_{F}} L_{F} \,ds \\
 =&  \mbox{sign}_{K}(F) |F|^{-1} \int_{F} \bsigma_0^{\Delta}  \cdot \bn_{F}  \,ds
  =  \mbox{sign}_{K}(F)r(\phi_{F}^{-}) 
 = r(\phi_{F}|_{K}).
\end{split}
\end{equation}
We can also observe that
\[
	\sum_{F \in \cE_{K}} {\phi_{F}|_{K}} \equiv 1 \quad \mbox{on } \partial K.
\]
Moreover, there exists a function $v_{K} \in \mbox{span}\left\{\phi_{i,K}\right\}_{i=1}^{m_k}$ such that
$\displaystyle\sum_{F \in \cE_{K}} {\phi_{F}|_{K}} + v_{K} \equiv 1$  in $\bar K$.
Finally, combining \cref{3.27} with the facts that $ r(v_{K})= \mathcal{A}(\bsigma_0^{\Delta} ,v_{K}) = 0$ yields
\begin{equation}
\mathcal{A}(\bsigma_0^{\Delta} ,1_{K}) = r(1_{K}) \quad \forall K \in \cT_{h},
\end{equation}
and, hence, \cref{local-eq}.

The proof for \cref{local__def-Aa} is immediate from \cref{nc-2} and the fact that
\[
\int_{K} \nabla \cdot  \tilde \bsigma_{0}^{nc} \,dx 
= \int_{K} \tilde \bsigma_{k-1}^{nc} \,dx \quad \forall K \in \cT_{h}.
\]
This completes the proof of the lemma.
\end{proof}

\begin{remark}
To the author's knowledge, the conservative flux  defined in \cref{local__def-A} is novel.
\end{remark}


\subsection{NC  FEM of order $2$:
 ($u_{h} \in NC(\cT_{h}, 2), 
\hat \bsigma_{h}^{nc} \in RT_f(\cT_h, 0)$)}

The NC spaces for even orders have dramatically different structures from those for odd orders. One critical challenge for even-order NC spaces is the non-existence of Gaussian-Legendre basis functions in 2D. Therefore, a direct localization similar to NC odd order is no longer available for even orders. Since the basis functions for even-order NC finite elements are defined differently depending on the order, we only apply the EARM to the second order  Fortin-Soulie NC FEM in this paper.

Define the following functions:
 \[
 	\phi_F = \underset{z \in \mathcal{N}_F}{ \Pi \lambda_{z}} \quad \mbox{and} 
	\quad \phi_K = \left(2 - 3\underset{z \in \mathcal{N}_K}{\sum} \lambda_z^2 \right)\Big|_K
 \]
 where $\lambda_z \in H^1(\O)$ is the barycentric piecewise linear basis function corresponding to $z \in \mathcal{N}$
 and $\cN_F (\cN_K)$ is the set of all vertices on $F (K)$.
 Here $\phi_K$ is the so-called bubble function that is only supported in $K$ and satisfies $\int_{\partial K}\phi_K p \,ds =0$ for all $p \in \mathbb{P}_1(K)$.
 The second order Fortin Soulie NC FEM \cite{fortin1983non} is defined as follows:
\[
	NC(\cT_h, 2) = \mbox{span}\{\lambda_z, \phi_F, \phi_K, z \in \cN_{h}, F \in \cE, K \in \cT_h \},
\]
where $\cN_{h}$ is the set of all verties on $\cT_{h}$.

We note that
{{\begin{equation}\label{space-relationship}
	\sum_{F \in \cE_K}\phi_F|_K + \phi_K \equiv 1. 
\end{equation}}}

\begin{lemma}\label{lem:nc-2}
Define  $\bsigma_0^{\Delta}  \in RT(\cT_h, 0)$ such that
\begin{equation}\label{local__def-B}
	\int_{F} \bsigma_0^{\Delta}  \cdot \bn_{F} \,ds =  6 \, r(\phi_{F}^{-}),
	\quad \forall F\in \cE \setminus \cE_{N},
\end{equation}
where $ r(v) = (f - \nabla \cdot \tilde \bsigma_{1}^{nc}, v)$. 
Then the recovered flux $\hat \bsigma_{h}: = \bsigma_0^{\Delta} +  \tilde \bsigma_{1}^{nc}$ satisfies
\[
\int_{K} \nabla \cdot \hat \bsigma_{h} \,dx = \int_{K} f \,dx  
\quad \forall K \in \cT_{h}.
\]
\end{lemma}
\begin{proof}
From  \cref{lem:existence}, it is sufficient to prove
\begin{equation}\label{local-eq-A}
\mathcal{A}(\bsigma_0^{\Delta} ,v) = r(v) \quad \forall  v \in  DG(\cT_h,0).
\end{equation}
For any $K \in \cT_{h}$, we first have that $\mathcal{A}(\bsigma_0^{\Delta} ,\phi_{K}) = r(\phi_K) =0$.
Moreover, using the fact that $\int_{F} \phi_{F} \,ds = \dfrac{1}{6} |F|$  and $\bsigma_0^{\Delta}  \cdot \bn_{F}$ is a constant, we have
\begin{equation}
\begin{split}
\mathcal{A}(\bsigma_0^{\Delta} ,\phi_{F}|_{K}) &
= \int_{F} \bsigma_0^{\Delta}  \cdot \bn_{F}
\jump{\phi_{F}|_{K}} \,ds = 
  \dfrac{1}{6} \mbox{sign}_{K}(F)\int_{F} \bsigma_0^{\Delta}  \cdot \bn_{F} \,ds \\
 =&\mbox{sign}_{K}(F)r(\phi_{F}^{-}) = r(\phi_{F}|_{K}),
\end{split}
\end{equation}
which, combined with \cref{space-relationship}, proves \cref{local-eq-A}. 
This completes the proof of \cref{local-eq-A} and hence the lemma.

\begin{remark}
This recovered flux defined in \cref{{lem:nc-2}} is equivalent to the projection of the fluxes introduced in \cite{kim2012flux} onto the \( RT(\cT_{h},0) \) space. 
 \end{remark}

\end{proof}

\section{Applying EARM for Conforming FEM}\label{sec:5}

Define the $k$-th order ($k \ge 1$) conforming finite element spaces by
\[
CG(\cT_h,k) = \{v\in H^1(\O)\,:\,v|_K\in \mathbb{P}_k(K),\quad\forall\,\,K\in\cT_h\}
\]
and
\[
CG_{0,\Gamma_D}(\cT_h,k) = H_{0,\Gamma_D}^1(\O) \cap CG(\cT_h,k).
\]
The conforming finite element solution of order $k$ is to find $u_k^{cg} \in CG_{0,\Gamma_D}(\cT_h,k)$ such that
{{\begin{equation} \label{CG-solution}
	(A \nabla u_k^{cg}, \nabla v)_{\O}
	=(f,v)_\O-\left<g, v\right>_{\Gamma_N}, \quad \forall \, v  \in CG_{0,\Gamma_D}(\cT_h,k).
\end{equation} }}

First we have that for any $v \in CG_{0,\Gamma_D}(\cT_h,k)$, by the definition of $ \tilde \bsigma_{k-1}^{cg}:= \tilde \bsigma_{k-1} (u_k^{cg})$ and integration by parts, there holds
 \begin{equation}
 \begin{split}
&(f,v) - (\gradt \tilde \bsigma_{k-1}^{cg},v)
 =
    (f,v)-  (A \nabla u_{k}^{cg}, \nabla v) +
    \sum_{F \in \cE}( \{ A \nabla u_{k}^{cg}\cdot \bn_{F}\}_{w}^{F}, \jump{v})_{F} \\
    &
  =
    (f,v)  -  (A \nabla u_k^{cg}, \nabla v) -
    < g, v>_{\Gamma_{N}}=0.
    \end{split}
 \end{equation}
 
 We immediately have the following compatibility result for the conforming finite element solution.
 \begin{lemma}
 Let $\mathcal{A}(\bftau_h,v) := \underset{F \in \cE \setminus \cE_{N}}{ \sum } \left<\bftau_h \cdot \bn_F, \jump{v} \right>_F$ be a bilinear form defined on $\bftau_h \in \mathring{RT}(\cT_h, s) \times DG(\cT_{h},k), \,  s \le k-1$.
There holds
 \begin{equation}\label{compatibility-cg}
\begin{split}
	&\mathcal{A}(\bftau_h,v)= r(v) = 0, \quad\forall v \in CG_{0,\Gamma_D}(\cT_h,k). \quad
\end{split}
\end{equation}
where $r(v) = (f - \nabla \cdot \bsigma_{k-1}^{cg}, v)_{\O}$.
 \end{lemma}
 
Similar to \cref{nc-compatbility},
we aim to solve $\bsigma_{s}^{\Delta} \in \mathring{RT}(\cT_{h},s)$ for the conforming finite element solutions such that
 \begin{equation}\label{-cg-problem}
\begin{split}
	&\mathcal{A}(\bsigma_{s}^{\Delta},v)=\underset{F \in \cE \setminus \cE_{N}}{ \sum } \left<\bsigma_{s}^{\Delta}, \cdot \bn_F, \jump{v} \right>_F = r(v) \quad\forall v \in DG(\cT_h,s). 
\end{split}
\end{equation}	
 
 \subsection{Deriving the equilibrated residual method}\label{sec:5.1}
 Before presenting the novel aspects of our approach, we demonstrate that the EARM can systematically recover the classical equilibrated residual method introduced in \cite{ainsworth2000posteriori}. This serves as a validation step, showing that our framework naturally reproduces well-established results. 
For direct comparison, we also consider the linear elements, i.e., $u_{1}^{cg} \in CG(\cT_{h},1)$ which is the same as in \cite{ainsworth2000posteriori}.
 
For each $z \in \cN$, define
$
	r_{z,K} = r(\lambda_{z}|_{K})
$
and
$\sigma_{z,F}^{\Delta} = \left<\bsigma_{s}^{\Delta} \cdot \bn_F, \lambda_{z}\right>_F$.
Testing $v = \lambda_{z}|_{K}$ in \cref{-cg-problem} for all $K \subset \omega_{z}$ yields
the following linear system in terms of the first-moment unknowns $\{\sigma_{z,F}^{\Delta}\}_{F \in \cE_{z}}$:
 \begin{equation}\label{local-monents}
\sum_{F \in \cE_{z} \cap \cE_{K} \setminus \cE_{N}} \mbox{sign}_{K}(F) \bsigma_{z,F}^{\Delta}  = r_{z,K},
\end{equation}
where $\cE_z$ denotes the set of all edges sharing $z$ and $\o_{z}$ is the star-patch including all elements sharing $z$ as a common vertex.
We note that the local patch problem in \cref{local-monents} is similar to (6.37) in \cite{ainsworth2000posteriori}, though derived differently.
More specifically,  the local patch problem in \cite{ainsworth2000posteriori} is equivalent to solving $\{\sigma_{z, F}\}_{F \in \cE_{z}}$, however, with a slightly more complex form, such that
 \begin{equation}
\sum_{F \in \cE_{z} \cap \cE_{K} \setminus \cE_{N}} \mbox{sign}_{K}(F) \bsigma_{z,F} =
 r_{z,K} + \sum_{F \in \cE_{z} \cap \cE_{K} \setminus \cE_{N}} \mbox{sign}_{K}(F) 
 \int_{F} \{ A \nabla u_{h}^{cg} \cdot \bn_{F}\}   \lambda_{z} \,ds.
\end{equation}
 This is because the recovered flux $\hat \bsigma_{h}$ is solved directly in \cite{ainsworth2000posteriori},
 whereas in our approach,  the local correction flux  $\bsigma_{k-1}^{\Delta}$ is  solved. 
 
From Table 6.1 in \cite{ainsworth2000posteriori}, there exist infinitely many solutions for the local system \cref{local-monents}, if $z$ is an interior vertex or if $z$ is on the boundary and shared by two facets $F \in \cE_{D}$. To resolve this, an extra constraint is imposed in \cite{ainsworth2000posteriori}, and then the Lagrangian multiplier method is used to solve the local patch problems. Finally, an assembling based on the first moments is required to solve for the recovered flux $\hat\bsigma_{h}$.
 
In the following subsections, we introduce two enhancements to the equilibrated residual method based on the EARM. These eliminate the need for constrained minimization via the Lagrangian method and bypass the final assembly step.
\subsection{Restricting $\bsigma_{s}^{\Delta}$ to the orthogonal complement of divergence-free space}\label{subsec:restriction}
Recall that \cref{nc-form} has infinitely many solutions given the kernel of divergence-free space.
In this subsection, we derive a  recovery method that applies a natural constraint, i.e., restricts the correction  in the orthogonal complement of divergence-free space. Such a constraint results in solving for the correction  in a different yet simple space without the need for a Lagrangian multiplier. 

Denote the divergence-free space as
 \[\mathring{RT}^{0}(\cT_{h},s) = \{\bftau \in \mathring{RT}(\cT_{h},s): \nabla \cdot \bftau =0  \}.\]
Also denote by $\mathring{RT}^{0}(\cT_{h},s)^{\perp}$ the orthogonal complement of the divergence free space $\mathring{RT}^{0}(\cT_{h},s)$.

%

  Now  for any $w \in DG(\cT_{h},s)$, define a mapping $\mathbf{S}:  w \in DG(\cT_{h},s) \to$
  $\mathbf{S}(w) \in \mathring{RT}(\cT_{h},s)$ such that
  \beq \label{dg-}
  \mathbf{S}(w) \cdot \bn_{F}|_{F} = A_{F} h_{F}^{-1}\jump{w}|_F \quad \forall F \in {\cE \setminus \cE_{N}},
  \eeq
  where $A_{F} = \min(\a_{F}^{+}, \a_{F}^{-})$ when $F \in \cE_{I}$ and $A_{F} =\a_{F}^{-}$ when $F \in \cE_{D}$.
\begin{lemma}\label{lem:perp-space}
The following relationship holds:
 \[
 	\{ \mathbf{S}(w): w \in DG(\cT_{h},s) \} \subset \mathring{RT}^{0}(\cT_{h},s)^{\perp}.
	\]
\end{lemma}
\begin{proof}
We first observe that for any $\bftau_h \in \mathring{RT}^{0}(\cT_{h},s)$, there holds
 \begin{equation}\label{kernel3}
0 = (\nabla  \cdot \bftau_h,  v)_{\cT_{h}}  \Leftrightarrow
 \underset{F \in \cE\setminus \cE_{N}}{ \sum } \left<\bftau_h \cdot \bn_F, \jump{v} \right>_F =\mathcal{A}(\bftau_h,v) =0 \quad \forall v \in DG(\cT_{h},s).
 \end{equation}
 Thus, $\bftau_h \in  \mathring{RT}^{0}(\cT_{h},s)^{\perp}$ if and only if
    \begin{equation}\label{sufficientCondition}
\mathcal{A}(\bftau_h,v) =0 \quad \forall v \in DG(\cT_{h},s) \Rightarrow \bftau_h \equiv 0.
 \end{equation}
Therefore, to prove that $\mathbf{S}(w) \in  \mathring{RT}^{0}(\cT_{h},s)^{\perp}$, it is sufficient to prove that $\mathbf{S}(w)$ satisfies \cref{sufficientCondition} for all $w \in DG(\cT_{h},s)$.
Assuming
   \begin{equation}
   \mathcal{A}(\mathbf{S}(w),v) =0 \quad \forall v \in DG(\cT_{h},s),
 \end{equation}
immediately yields
    \begin{equation}
   \mathcal{A}(\mathbf{S}(w),w)  =0,
 \end{equation}
i.e., $\|A_{F}^{1/2} h_{F}^{-1/2}\jump{w}\|_{\cE\setminus \cE_{N}} = 0$, hence, $\| \mathbf{S}({w}) \cdot \bn_{F} \|_{\cE\setminus \cE_{N}} \equiv 0$. This completes the proof of the lemma.
\end{proof}

Inspired by \cref{lem:perp-space}, we reform the problem \cref{nc-form} by restricting the  in the space of $\mathring{RT}^{0}(\cT_{h},s)^{\perp}$   as following:
finding $u_s^{\Delta} \in DG(\cT_h, s) $ such that
\begin{equation}\label{b-correction}
	\mathcal{A}( \jump{u_s^{\Delta}}, v)
= r(v) \quad \forall \, v \in DG(\cT_h,s),
\end{equation}
where
\[
\mathcal{A}( \jump{u_s^{\Delta}}, \jump{v}) 
:=\sum_{F \in \cE \setminus \cE_{N}} \int_{F}  A_{F}h_{F}^{-1} \jump{u_s^{\Delta}} \jump{v} \,ds.
\]

Note that only the information of $\jump{u_{s}^{\Delta}}$ on facets $\cE \setminus \cE_{N}$ is used in the formulation. We further
define for any $s \ge 1$ the quotient spaces:
\[
	 DG^{0}(\cT_h, s) = \{v \in DG(\cT_h, s): v|F=0 \, \forall F \in \cE_{N} \} / \{v \in  CG(\cT_h, s), v|_{\cE} =0\}.
\]
Here we have used $A/B$ to denote the quotient space of $A$ by $B$.
 \cref{b-correction} is then equivalent to finding $u_s^{\Delta} \in DG^{0}(\cT_h, s) $ such that
\begin{equation}\label{b-correction-A} 
	\mathcal{A}( \jump{u_s^{\Delta}}, \jump{v})_{F}
= r(v) \quad \forall \, v \in DG^{0}(\cT_h,s).
\end{equation}

One feature of \cref{b-correction-A} is the well-posedness thanks to the same trial and test spaces along with continuity and coercivity in the space of $DG^{0}(\cT_h,s)$.

\begin{lemma}
	\cref{b-correction-A} has a unique solution $u_s^{\Delta} \in DG^{0}(\cT_h, s)$ for all integers $0 \le s \le k-1$.
\end{lemma}
\begin{proof}
We first have that the bilinear form $	\mathcal{A}( \jump{u_s^{\Delta}}, \jump{v})_{F}$ is continuous and coercive in the space of $DG^{0}(\cT_h, s)$ under the norm:
\beq \label{tri-norm}
\tri v \tri =  \sqrt{\sum_{F \in \cE\setminus \cE_{N}} h_{F}^{-1}\| A_{F}^{1/2}\jump{v}\|_{F}^{2}}.
\eeq Then, by the Lax-Milgram theorem, \cref{b-correction-A}  has a unique solution for all $0 \le s \le k$. 
\end{proof}
\begin{remark}
When $s=0,$ our method coincides with the equilibrated flux  recovered in \cite{odsaeter2017postprocessing}.
	Though the solution of \cref{b-correction-A} is unique, it renders a global problem. However, the number of degrees of freedom for the global problem is much less than the original problem, as only the degree of freedom on the boundary is needed.
	
Moreover, for \(1 \leq s \leq k-1\), we can localize it using the same techniques as in \cite{becker2016local} by replacing the exact integrals, \(\int_{F} A_{F}h_{F}^{-1} \jump{u_s^{\Delta}} \jump{v} \,ds\), with an inexact Gauss-Lobatto quadrature. In two dimensions, this involves employing \(s\) points on each edge, which ensures exactness up to order \(2s-1\).
With Gauss-Lobatto quadrature in two dimensions, our recovered flux coincides with that in \cite{becker2016local}. Furthermore, the recovered flux can be computed entirely explicitly, as shown in \cite{capatina2024robust}.

\end{remark}

\begin{lemma}\label{lem:sigma-hat-1}
Define
\begin{equation}\label{recover--cg}
	\hat \bsigma_h^{cg} = \tilde \bsigma_{k-1}^{cg} + \bsigma_s^{\Delta},
\end{equation}
where $\bsigma_s^{\Delta} =\mathbf{S}({u_{s}^{\Delta}})$ in which $u_{s}^{\Delta}$ is the solution to \cref{b-correction-A} for some $ 0 \le s \le k-1$.
The recovered flux defined in \cref{recover--cg} satisfies
	\begin{equation}
		(\nabla \cdot \hat \bsigma_h^{cg}, v) = (f,v) \quad \forall v \in DG(\cT_{h},s).
	\end{equation}
\end{lemma}
\begin{proof}
	It is sufficient to prove the local conservation for any $v \in \mathbb{P}_{s}(K) \subset DG(\cT_h,s)$.
	By the definitions, integration by parts, \cref{b-correction-A} 
	\begin{equation}
	\begin{split}
		( \nabla \cdot \hat\bsigma_h^{dg}, v)_{K} 
		&= 
		(\nabla \cdot  \tilde\bsigma_{k-1}^{cg}, v)_{K} +
		(\nabla \cdot   \bsigma_h^{\Delta}, v)_{K} \\
		&= 
		(\nabla \cdot  \tilde\bsigma_{k-1}^{cg}, v)_{K} +
		( \bsigma_h^{\Delta},  \nabla v)_{K} +
		(  \bsigma_h^{\Delta} \cdot \bn_{F}, \jump{v})_{\partial K} \\
		&= 
		(\nabla \cdot  \tilde\bsigma_{k-1}^{cg}, v)_{K} +
		\sum_{F \in \cE_{K} \setminus \cE_{N}}(  \jump{u_{s}^{\Delta}}, \jump{v})_{F} 
		=  (\nabla \cdot  \tilde\bsigma_{k-1}^{cg}, v)_{K} + r(v) \\
		&= (\nabla \cdot  \tilde\bsigma_{k-1}^{cg}, v)_{K}  +  (f,v)_K
		- (\gradt \tilde \bsigma_{k-1}^{cg},v)_{K}  =  (f,v)_{K}.
		\end{split}
	\end{equation}
	We have completed the proof of the lemma.
\end{proof}

\subsection{Localization via Partial POU}\label{subsec:localization}

In this subsection, we propose a patch localization in which non-constrained local problems in star patches must be solved for the conforming methods. 

Define
\[
	\mathring{RT}_{0,\partial \o_z}(\omega_z,s) = \{ \btau_h \in \mathring{RT}(\omega_z,s),  \btau_h \cdot \bn_F = 0 \, \forall F \in \partial \omega_z. \}
\]

Firstly we define local problems: finding  $\bsigma_z^{\Delta}  \in \mathring{RT}_{0,\partial \o_z}(s, \omega_z)$, $0 \le s \le k-1$, for each $z \in \cN$ such that

\begin{equation}\label{local-pro-z}
\mathcal{A}(\bsigma_z^{\Delta} ,v) 
	= r(v \lambda_z) \quad \forall v \in DG(\o_{z},s),
\end{equation}
where
\[
\mathcal{A}(\bsigma_z^{\Delta} ,v) = \sum_{F \in \cE_{z} \setminus \cE_{N}}
\int_{F} \bsigma_z^{\Delta} \cdot \bn_{F} \jump{v} \,ds,
\] 
and
$ r(v \lambda_z)  = (f - \nabla \cdot \tilde \bsigma_{k-1}^{cg}, v \lambda_{z })_{\O}$.

We then define 
\begin{equation}\label{-POU-a}
\bsigma_s^{\Delta} =  \sum_{z \in \cN} \bsigma_z^{\Delta}.
\end{equation}

\begin{lemma}\label{lem:local-unity}
The flux defind in \cref{-POU-a} satisfies
\begin{equation}
\mathcal{A}( \bsigma_s^{\Delta}, v) =  r(v) \quad \forall v \in DG(\cT_{h},s).
\end{equation}
\end{lemma}
\begin{proof}
It is sufficient to prove that for any $K \in \cT_{h}$, there holds
\begin{equation}
\mathcal{A}( \bsigma_s^{\Delta}, v) =  r(v) \quad \forall v \in DG(K,s).
\end{equation}
	Since $v \in DG(K,s)$ implies that $\mathcal{A}\left( \bsigma_z^{\Delta}, v\right) =0$ for all $z \not \in \cN_{K}$, we have
	\begin{equation}
	\begin{split}
\mathcal{A}( \bsigma_s^{\Delta}, v) =  
 \sum_{z \in \cN_{K}}   \mathcal{A}\left( \bsigma_z^{\Delta}, v\right).
\end{split}
\end{equation}
Obviously, we also have $v \in DG(\o_{z},s)$ for all  $z \in \cN_{K}$, and hence,
	\begin{equation}
	\begin{split}
\mathcal{A}( \bsigma_s^{\Delta}, v) =  
 \sum_{z \in \cN_{K}}   \mathcal{A}\left( \bsigma_z^{\Delta}, v\right) =
  \sum_{z \in \cN_{K}}   r(v \lambda_z) =
 r(v).
\end{split}
\end{equation}
This completes the proof of the lemma.
\end{proof}

\begin{remark}
	\cref{local-pro-z} is a slight variation of the classical equilibrated residual method for which $r(v \lambda_{z})$ is replaced by $r(v)$. The advantage of using \cref{local-pro-z} is that it  solves the  directly instead of moments of the  in the classical equilibrated residual method, avoiding the final assembling.
\end{remark}
We now analyze the local patch problem in \cref{local-pro-z} specifically for $s=1, k \ge 2$ in the two dimensions.
The results of the other cases in two dimensions are similar.

\begin{lemma}\label{lem:local-unity-solutions}
When $z \in \cN_{I}$ or $z \in \Gamma_{D}$ and shared by two Dirichlet boundary edges,
the local problem in \cref{local-pro-z}  for $s=1$ has infinitely many solutions, with a one-dimensional kernel.
When $z \in \partial \O$ and $z$ is shared by at least one boundary edge on the Neumann boundary, \cref{local-pro-z} has a unique solution.
\end{lemma}
\begin{proof}
	
	We firstly analyze the case when $z$ is an interior vertex, i.e., $z \in \cN_{I}$. 
	We denote by $\cT_z$ the set of elements sharing $z$ as a common vertex in $\o_z$. We also let $|\cT_z|$  and $|\cE_z|$be the number of elements and edges in 
	$|\cT_z|$  and $|\cE_z|$, respectively.
	We note that  $\mbox{dim}(DG(\o_{z},1)) =3 |\cT_z|$, $\mbox{dim}(\mathring{RT}_{0,\partial \o_z}(\omega_z,s)) =2 |\cE_z| $ and $|\cT_z| = |\cE_z|$  in this case. 
	We firstly test $v = \lambda_z|_K$ for each $K \in \cT_z$
	in \cref{local-pro-z}, which yields $|\cT_z|$ equations:
		\begin{equation}
 \mathcal{A}(\bsigma_z^{\Delta} ,\lambda_z|_K) 
=r(v \lambda_z|_{K}).
\end{equation}
	However, we also have the following compatibility equation when $k \ge 2$:
	\begin{equation}
\sum_{K \in \cT_z} \mathcal{A}(\bsigma_z^{\Delta} ,\lambda_z|_K) 
= \mathcal{A}(\bsigma_z^{\Delta} ,\lambda_z) 
	= r(\lambda_z^2) =0.
\end{equation}
The last equation holds since $\lambda_z^2 \in CG(\cT_{h}, k\ge 2)$.
Therefore, testing on $\lambda_z|_K$ for each $K \in \cT_z$ yields $|\cE_z| -1$ independent linear equations.

We now consider each vertex $z' \in \partial \o_z \cap \cN$.
Let $K_{z'}^{+}$ and $K_{z'}^{-}$ denotes the two elements sharing $z'$ and $z$ as common vertices.
Similarly, we can test in \cref{local-pro-z} with $v = \lambda_{z'}|_K$ with $K = K_{z'}^{+}$ or $K = K_{z'}^{-}$ which yields two equations.
However, we also have the following compatibility:
	\begin{equation}
 \mathcal{A}(\bsigma_z^{\Delta} ,\lambda_{z'}|_{ K_{z'}^{-}})  +
  \mathcal{A}(\bsigma_z^{\Delta} ,\lambda_{z'}|_{ K_{z'}^{+}}) =
    \mathcal{A}(\bsigma_z^{\Delta} ,\lambda_{z'}) =0
  	\end{equation}
	and
	\begin{equation}	  
	 r(\lambda_{z'}|_{ K_{z'}^{-}} \lambda_z)  +  r(\lambda_{z'}|_{ K_{z'}^{+}} \lambda_z)
	=  r(\lambda_{z'} \lambda_z)   =0.
	\end{equation}
Thus on each $z' \in \partial \o_z \cap \cN$ there yields one linear equation for 
$\bsigma_z^{\Delta}$, i.e.,
		\begin{equation}
 \mathcal{A}(\bsigma_z^{\Delta} ,\lambda_{z'}|_{K_{z'}^{-}}) 
=r(\lambda_{z} \lambda_{z'}|_{K_{z'}^{-}}) \quad \forall z' \in \partial \o_z \cap \cN.
\end{equation}
Note that the number of vertices on $\partial \o_z$ is equal to $|\cE_z|$.
Therefore, in total, we have $2|\cT_z|-1$ linearly independent equations. In summary, with $2|\cE_z|$ unknowns, the linear local problem in \cref{local-pro-z} has a one-dimensional kernel for $z \in \cN_{I}$. 

Considering the case when $z \in \partial \O$ and $\partial \o_{z}
\subset \Gamma_{N}$. Note that  $|\cT_z| +1 = |\cE_z|$.  Since $\bsigma_z^{\Delta} \cdot \bn_{F} =0$ on $\Gamma_{N}$,
we have that $\mbox{dim}(\mathring{RT}_{0,\partial \o_z}(\omega_z,s)) =2 (|\cE_z| -2) $. Similarly, testing  $v = \lambda_z|_K$ for each $K \in \cT_z$ yields  $|\cT_z|-1 = |\cE_{z}| -2$ equations. Testing  $v = \lambda_{z'}|_{K_{z'}^{-}}$ for each $z' \in \partial (\o_z \cap \cN) \setminus \Gamma_N$ yields  another linearly independent
$ |\cE_z| -2$ equations. Therefore, in this case, there is a unique solution.

Considering the case when $z \in \cE_D$, i.e., $z \in \partial \O$ and $z$ is shared by two facets in $\cE_{D}$. 
Similarly, testing  $v = \lambda_z|_K$ for each $K \in \cT_z$ yields  $|\cT_z| = |\cE_z|-1$ equations. There is no compatibility for this case. Testing  $v = \lambda_{z'}|_{K_{z'}^{-}}$ 
for each $z \in \cN \cap \partial \o_z$ yields 
$ |\cE_z|$ equations. Therefore,   with $2|\cE_z|$ unknowns,  there is also one-dimensional kernel in this case.

Finally, when $z \in \partial \O$, and $z$ is shared by a Dirichlet facet and a Neumann facet,  we can also prove that  \cref{local-pro-z} has a unique solution by a similar proof.
This completes the proof of the lemma.
\end{proof}

\begin{theorem}
	When the solution of \cref{local-pro-z} is not unique, let $
	\bsigma_z^{\Delta}  \in \mathring{RT}_{0,\partial \o_z}(s, \omega_z)$ be a specific solution to \cref{local-pro-z}  and let
	$\bsigma_z^{\#} \in \mathring{RT}_{0,\partial \o_z}(s, \omega_z)$ be a nontrivial solution that satisfies 
	\begin{equation}\label{local-pro-z0}
\mathcal{A}(\bsigma_z^{\#} ,v) 
	= 0 \quad \forall v \in DG(\o_{z},s)
	\quad
	\mbox{and} \quad
	\|A^{-1/2} \bsigma_{z}^{\#}\|_{\o_{z}}=1.
\end{equation}
	Define $ \bsigma_z^{\Delta,*}$ such that
	 \begin{equation}\label{min-property}
	 \bsigma_z^{\Delta,*}   =\underset{\lambda \in \mathbb{R}}{\operatorname{argmin}}
	 \| A^{-1/2} ( \bsigma_z^{\Delta} - \lambda \bsigma_z^{\#})  \|_{\o_{z}}.
	 \end{equation}
	 Then 
	 \begin{equation}
\mathcal{A}( \bsigma_s^{\Delta,*}, v) =  r(v) \quad \forall v \in DG(\cT_{h},s),
\end{equation}
where
\begin{equation}\label{-cg-null-correction}
\bsigma_s^{\Delta,*} =  \sum_{z \in \cN^{*}} \bsigma_z^{\Delta,*} + 
\sum_{z \in \cN \setminus \cN^{*}} \bsigma_z^{\Delta} 
=  \sum_{z \in \cN^{*}} ( \bsigma_z^{\Delta} - \lambda_{z} \bsigma_z^{\#}) 
+
\sum_{z \in \cN \setminus \cN^{*}} \bsigma_z^{\Delta},
\end{equation}
$\cN^{*}$ is the set of interior vertices and vertices on the boundary sharing by two Dirichlet boundaries, and $\lambda_{z} ={(A^{-1}  \bsigma_z^{\Delta},  \bsigma_z^{\#})_{\o_{z}}}$.
\end{theorem}
\begin{proof} The theorem follows directly from the results established in \cref{lem:local-unity} and \cref{lem:local-unity-solutions}. \end{proof}
\begin{remark}
Even though $\bsigma_z^{\Delta,*}$ satisfies the minimization property in  \cref{min-property}, the solution can be sought explicitly as in \cref{-cg-null-correction}. 
In \cref{appen-pouprogramming}, we also show that the solution for $\bsigma_s^{\Delta,*} $ can be explicitly computed in two dimensions.
\end{remark}


\begin{lemma}\label{lem:sigma-hat-cg2}
Define
\begin{equation} 
	\hat \bsigma_h^{cg} = \tilde \bsigma_{k-1}^{cg} + \bsigma_s^{\Delta,*} ,
\end{equation}
where $\bsigma_h^{\Delta,*}$ is  defined in \cref{-cg-null-correction}.
The recovered flux defined in \cref{recover--cg} satisfies
	\begin{equation}
		(\nabla \cdot \hat \bsigma_h^{cg}, v) = (f,v) \quad \forall v \in DG(\cT_{h},s).
	\end{equation}
\end{lemma}
\begin{proof}
	The proof is similar to \cref{lem:sigma-hat-1}.
\end{proof}

\section{Automatic Global Reliability}\label{sec:6}
We first cite the following reliability result proved in \cite{cai2021generalized}.
 \begin{theorem}\label{PS:2d}
Let $u\in H^1_{D}(\O)$ be the solution of  {\em (\ref{pde})}. In two and three dimensions,
 for all $w \in H^1(\cT_{h})$, we have
\[
	 \|A^{1/2}\nabla_h(u-w)\|^2 =
	 \inf_{\btau\in \S_f(\O)}\| A^{-1/2}\btau+
	 A^{1/2}\nabla_h w\|^2 + \inf_{v\in H^1_{D}(\O)}\|A^{1/2}\nabla_h(v-w)\|^2.
\]
 \end{theorem}
 where
  \[
\S_f(\O) = \Big\{ \btau \in H(\divvr;\O) : \gradt \btau =f  \mbox{ in } \O \;\mbox{ and }
  \; \btau \cdot \bn = {g  \mbox{ on } \Gamma_N} \Big\}.
  \]
  Based on the theorem and Corollary~3.5 in \cite{cai2021generalized}, 
the construction of an equilibrated a posteriori
error estimator for discontinuous finite element solutions is reduced to recover an equilibrated 
in $\S_f(\O)$ and to recover 
either a potential function in $H^1(\Omega)$ or a curl free vector-valued
function in $H(\curll;\Omega)$. In this paper, we focus on the part of equilibrated flux recovery.
We note that $  \displaystyle\inf_{\btau\in \S_f(\O)}\| A^{-1/2}\btau+A^{1/2}\nabla_h w\|^2 $ is usually referred to as the conforming error of $w$.

Note that in our case, the recovered flux $\hat \bsigma_{h}$ lies in $RT_f(\cT_h,s)$. A similar result to \cref{PS:2d} can be easily proved with an additional term $c\|f - f_{s}\|_{\O}$ on the right, 
\[
	 \|A^{1/2}\nabla_h(u-w)\|^2 \le 
	\| A^{-1/2}\hat \bsigma_{h}+
	 A^{1/2}\nabla_h w\|^2 + \inf_{v\in H^1_{D}(\O)}\|A^{1/2}\nabla_h(v-w)\|^2 + c\|f - f_{s}\|^{2}_{\O}.
\]
This additional oscillation term is of higher order when $f|_{K}$ is smooth for all $K \in \cT_{h}$. Therefore, it is usually neglected in the computation.

\section{Robust Efficiency} \label{sec:7}
In this section, we will prove the robust efficiency of the error indicator defined as the energy norm of the difference between the numerical and equilibrated recovered flux. Here the robustness refers to the independence of the efficiency constant with respect to the jump of the coefficient $A(x)$.

In particular, we define the local indicator and the global estimator for the conforming error by
\begin{equation} \label{estimators:cf_l}
	 \eta_{\sigma,K} = \|A^{-1/2} (\hat \bsigma_h - \bsigma_{h}) \|_K  
	\end{equation}
and
\begin{equation} \label{estimators:cf_g}
	\eta_\sigma =\left( \sum_{K \in \cT_h} \eta^2_{\sigma,K} \right)^{1/2}
	=\|A^{-1/2} (\hat \bsigma_h-  \bsigma_h) \|_{0,\O},
\eeq 
where 
$ \bsigma_h = - A\nabla u_{h}$ is the numerical flux and $\hat \bsigma_h$ is the recovered equilibrated flux.

We aim to prove the following local efficiency result:
	\begin{equation}\label{local-efficiency}
		\eta_{\sigma,K} 
	 \le C   \| A^{1/2} \nabla (u - u_h)\|_{\o_{K}} + \mbox{osc}(f),
\end{equation}
where $u_{h}$ is the finite element solution, $\o_{K}$ is a local neighborhood of $K$, $\mbox{osc}(f)$ is a oscillation term depends on the smoothness of $f$, and $C$ is independent of the jump of $A(x)$.


In the remaining section, we provide proof of efficiency for the newly developed fluxes in this paper. For the recovering fluxes that coincide with the existing state-of-the-art results, we refer to references \cite{Ai:07b,ern2007accurate, cai2021generalized, ainsworth2000posteriori}. For simplicity, we assume that the diffusion coefficient $A(x)$ is a piecewise constant function and that $A_{F}^{-} \le A_{F}^{+}$ for all $F \in \cE_{I}$. From here to thereafter, we use $a \lesssim b$ to denote that  that 
$a \le C b$ for a generic constant that is independent of the mesh size and the jump of $A$.

\subsection{Local efficiency for the flux recoveries of the NC FEM}

\begin{lemma}\label{lem:effi-nc-1}
Define the local error indicator
\begin{equation} \label{nc-local-indicator}
	 \eta_{\sigma,K} = \|A^{-1/2} \hat \bsigma_h + A^{1/2} \nabla u_k^{nc} \|_K 
	\end{equation}
	where $ \hat\bsigma_{h}$ is defined in \cref{lem:nc-case2}.
	Then we have
	\begin{equation}\label{local-effi-nc}
	 \eta_{\sigma,K} \le
	 C  \| A^{1/2} \nabla (u - u_{k}^{nc})\|_{\o_{K}} + \mbox{osc}(f)
\end{equation}
where $\o_{K}$ is a local neighborhood of $K$, $C$ is independent of the mesh size and the jump of $A$,  and $\mbox{osc}(f)$ is a oscillation term depends on the smoothness of $f$.
\end{lemma}	
\begin{proof}
We first observe that by triangle inequality,
\beq\label{tiangle-inequality}
\begin{split}
 \eta_{\sigma,K} &= \|A^{-1/2} (\hat \bsigma_h - \tilde \bsigma_{k-1}^{nc}) \|_K 
 +
 \|A^{-1/2}\tilde \bsigma_{k-1}^{nc} + A^{1/2} \nabla u_k^{nc}  \|_K\\
 &= \|A^{-1/2} \bsigma_0^{\Delta}\|_K 
 +
 \|A^{-1/2}\tilde \bsigma_{k-1}^{nc} + A^{1/2} \nabla u_k^{nc}  \|_K.
 \end{split}
\eeq
To bound $ \|A^{-1/2}\tilde \bsigma_{k-1}^{nc} + A^{1/2} \nabla u_k^{nc}  \|_K$, we first have by the equivalence of norms, \cref{weight1} and \cref{rt:1:aa},
\beq\label{bound-part1}
\begin{split}
	&\| A^{-1/2}\tilde \bsigma_{k-1}^{nc} + A^{1/2} \nabla u_k^{nc}  \|_K 
	\lesssim
	\sum_{F \in \cE_{K}}\| A_{K}^{-1/2}(\tilde \bsigma_{k-1}^{nc} 
	+ A_{K} \nabla u_k^{nc} ) \cdot \bn_{F}\|_{F}\\
	&
	\lesssim
	\sum_{F \in \cE_{K}}\| A_{F,max}^{-1/2}\jump{A \nabla u_{k}^{nc}} \cdot \bn_{F}\|_{F},
\end{split}
\eeq
where $A_{F,max} = \max(A_{F}^{-},A_{F}^{+})$.
To bound $ \|A^{-1/2}  \bsigma_0^{\Delta}\|_K$, recall	
\begin{equation}
	\int_{F} \bsigma_{0}^{\Delta}  \cdot \bn_{F} \,ds = |F| r(\phi_{F}^{-}),
	\quad \forall F\in \cE \setminus \cE_{N}.
\end{equation}
Since $ \bsigma_{0}^{\Delta} \in RT(\cT_{h},0)$, we have 
\beq\label{effi:1-a}
	\|  \bsigma_{0}^{\Delta} \|_K \le C\,
	\sum_{F \in \cE_K} h_F^{1/2} \|\bsigma_{0}^{\Delta}\cdot \bn_F \|_{F} 
	=
	C\,
	\sum_{F \in \cE_K} h_F  |r(\phi_{F}^{-})|.
\eeq
By the definition of $ \tilde \bsigma_{k-1}^{nc}$, integration by parts, and Cauchy Schwartz inequality, we have
\begin{equation}\label{effi:1-b}
\begin{split}
&r(\phi_{F}^{-}) 
= (f - \nabla \cdot  \tilde \bsigma_{k-1}^{nc}, \phi_{F}^{-})_{K_{F}^{-}}\\
&= (f,  \phi_{F}^{-}) + (\tilde \bsigma_{k-1}^{nc},  \nabla \phi_{F}^{-})_{K_{F}^{-}} 
-
<(\tilde \bsigma_{k-1} \cdot \bn_{K},  \phi_{F}^{-}) >_{\partial K_{F}^{-}}
\\
&= (f,  \phi_{F}^{-})_{K_{F}^{-}} - (A \nabla u_{k}^{nc},  \nabla \phi_{F}^{-})_{K_{F}^{-}} 
+
\sum_{F' \in \cE_{K}}< \{ A \nabla u_{k}^{nc} \cdot \bn_{F'}\}_w^F , \jump{ \phi_{F}^{-}}) >_{F'}
\\
&= (f + \nabla \cdot A \nabla u_{k}^{nc}, \phi_{F}^{-})_{K_{F}^{-}}
+
< \{ A \nabla u_{k}^{nc} \cdot \bn_{F}\}_w^F -   A \nabla u_{k}^{nc} \cdot \bn_{F} ,  \jump{\phi_{F}^{-}}) >_{F}
\\
&= (f + \nabla \cdot A \nabla u_{k}^{nc}, \phi_{F}^{-})_{K_{F}^{-}}
-
w_{F}^{+}< \jump{ A \nabla u_{k}^{nc} \cdot \bn_{F}},  {\phi_{F}^{-})} >_F \\
&\le 
 \| f + \nabla  \cdot A \nabla u_{k}^{nc} \|_{0, {K_{F}^{-}}} \| \phi_{F}^{-}\|_{0, {K_{F}^{-}}}+
w_{F}^{+}\| \jump{ A \nabla u_{k}^{nc} \cdot \bn_{F}} \|_{F}  \|{\phi_{F}^{-}}\|_{F}.
\end{split}
\end{equation}

Combining \cref{effi:1-b}  with \cref{effi:1-a} and the facts that,
\[
	\| \phi_{F}^{-}\|_{0,{K_{F}^{-}}} \lesssim 1 \quad \mbox{and} \quad 
	\|{\phi_{F}^{-}}\|_{F} \lesssim h_{F}^{-1/2}
\]
yields
\begin{equation}\label{bound-part2}
\begin{split}
	&\|  A^{-1/2}\bsigma_{0}^{\Delta} \|_K \le 	C\,
	\sum_{F \in \cE_K} h_F A_{K}^{-1/2}  |r(\phi_{F}^{-})| \\
&\lesssim 
 \sum_{F \in \cE_{K}} h_F\| A^{-1/2} (f + \nabla \cdot A \nabla u_{k}^{nc}) \|_{0,{K_{F}^{-}}} +
\sum_{F \in \cE_{K}}\dfrac{h_{F}^{1/2}}{A_{F,max}^{1/2}}\| \jump{ A \nabla u_{k}^{nc} \cdot \bn_{F}} \|_{F}.
\end{split}
\end{equation}
From the classical efficiency results, we also have
 \begin{equation}\label{efficiency-ele-residual}
  h_{K}\| A^{-1/2}(f + \nabla \cdot A \nabla u_{k}^{nc}) \|_K \lesssim \| A^{1/2} \nabla (u - u_{k}^{nc})\|_{\o_{K}} + \mbox{osc}(f)
 \end{equation}
 and
 \begin{equation}\label{efficiency--jump}
 \|A_{F,max}^{-1/2} \jump{ A \nabla u_{k}^{nc} \cdot \bn_{F}} \|_{F} \lesssim
  \| A^{1/2} \nabla (u - u_{k}^{nc})\|_{\o_{F}} + \mbox{osc}(f)
 \end{equation}
 where $\o_{K}$ and $\o_{F}$ are some local neighborhood of $K$ and $F$, respectively; and the involved constant does not depend on the jump of $A$.
Combining with \cref{bound-part1}, \cref{bound-part2} and the above two estimates, yields \cref{local-effi-nc}. This completes the proof of the lemma.
 \end{proof}
 
 \begin{lemma}
Define the local error indicator
\begin{equation} \label{nc-local-indicator-2}
	 \eta_{\sigma,K} = \|A^{-1/2} \hat \bsigma_h + A^{1/2} \nabla u_k^{nc} \|_K 
	\end{equation}
	where $ \hat\bsigma_{h}$ is defined in \cref{lem:nc-2}.
	Then we have
	\begin{equation}\label{local-effi-nc-2}
	 \eta_{\sigma,K} \le
	C   \| A^{1/2} \nabla (u - u_{k}^{nc})\|_{\o_{K}} + \mbox{osc}(f)
\end{equation}
where $\o_{K}$ is a local neighborhood of $K$, $\mbox{osc}(f)$ is a oscillation term depends on the smoothness of $f$, and the constant $C$ is independent of mesh size and the jump of $A$.
\end{lemma}
\begin{proof}
Similar to the proof of \cref{lem:effi-nc-1}, it is sufficient to prove
	\begin{equation}\label{local-effi-nc-2a}
	\|  A^{-1/2}\bsigma_{0}^{\Delta} \|_K \le
	C   \| A^{1/2} \nabla (u - u_{k}^{nc})\|_{\o_{K}} + \mbox{osc}(f).
\end{equation}
Since $ \bsigma_{0}^{\Delta} \in RT(\cT_{h},0)$, we have 
\beq \label{effi:1-aa}
	\|  \bsigma_{0}^{\Delta} \|_K \le C\,
	\sum_{F \in \cE_K} h_F^{1/2} \|\bsigma_{0}^{\Delta}\cdot \bn_F \|_{F} 
	=
	C\,
	\sum_{F \in \cE_K}   |r(\phi_{F}^{-})|.
\eeq
Similar  to \cref{effi:1-b}, we have
\begin{equation}\label{effi:1-bb}
\begin{split}
r(\phi_{F}^{-}) 
&= (f - \nabla \cdot  \tilde \bsigma_{1}^{nc}, \phi_{F}^{-})_{K_{F}^{-}}\\
&= (f + \nabla \cdot A \nabla u_{k}^{nc}, \phi_F^{-})_{K_{F}^{-}}
-
w_{F}^{+}< \jump{ A \nabla u_{k}^{nc} \cdot \bn_{F}},  {\phi_F^{-})} >_F \\
&\le 
 \| f + \nabla  \cdot A \nabla u_{k}^{nc} \|_{0, {K_{F}^{-}}} \| \phi_F^{-}\|_{0, {K_{F}^{-}}}+
w_{F}^{+}\| \jump{ A \nabla u_{k}^{nc} \cdot \bn_{F}} \|_{F}  \|{\phi_F^{-}}\|_{F}.
\end{split}
\end{equation}
Combining \cref{effi:1-bb}, \cref{effi:1-aa}, \cref{weight1} with the facts that,
\[
	\| \phi_{F}^{-}\|_{0,{K_{F}^{-}}} \lesssim h_{F} \quad \mbox{and} \quad 
	\|{\phi_{F}^{-}}\|_{F} \lesssim h_{F}^{1/2}
\]
yields
\begin{equation*}
\begin{split}
	&\|  A^{-1/2}\bsigma_{0}^{\Delta} \|_K \le 	C\,
	\sum_{F \in \cE_K} h_F A_{K}^{-1/2}  |r(\phi_{F}^{-})| \\
&\lesssim 
 \sum_{F \in \cE_{K}} \| A^{-1/2} (f + \nabla \cdot A \nabla u_{k}^{nc}) \|_{0,{K_{F}^{-}}} +
\sum_{F \in \cE_{K}}h_{F}^{1/2} A_{F,max}^{-1/2}\| \jump{ A \nabla u_{k}^{nc} \cdot \bn_{F}} \|_{F}.
\end{split}
\end{equation*}
Finally, applying the above result, \cref{efficiency--jump} and \cref{efficiency-ele-residual} yields \cref{local-effi-nc-2a}, and, hence, \cref{local-effi-nc-2}. This completes the proof of the lemma.
\end{proof}

\subsection{Local Efficiency of Flux Recovery in the Orthogonal Complement Space for Conforming FEM} 
For simplicity, we restrict our analysis in the two dimensions.
To show that the efficiency constant for the conforming case is independent of the jump of $A(x)$,
as usual, we assume that the distribution of the coefficients $A_K$
for all $K\in \cT$ is locally quasi-monotone \cite{petzoldt2002posteriori}, which is
slightly weaker than Hypothesis 2.7 in \cite{bernardi2000adaptive}.
The assumption has been used in the literature and remains essential for rigorous theoretical proofs. However, numerical experiments and practical applications indicate that this assumption is not necessary in computations, as the method performs reliably even when it does not strictly hold, e.g., see \cite{cai2017improved}.

For convenience of readers, we restate the definition of quasi-monotonicity.
Let $\o_z$ be the union of all elements having $z$ as a vertex.
For any $z\in\cN$, let
 \[
 \hat{\omega}_z=\{K\in\omega_z \,:\, A_K = \max_{K'\in\omega_z} A_{K'}\}.
\]
\begin{definition}\label{defnquasimonotone}
Given a vertex $z \in \cN$, the distribution of the coefficients $A_K$, $K\in\omega_z$,
is said to be {\em quasi-monotone} with respect to the vertex $z$
if there exists a subset $\tilde{\o}_{K,z,qm}$ of $\omega_z$ such that the union
of elements in $\tilde{\o}_{K,z,qm}$ is a Lipschitz domain and that
\begin{itemize}
\item if $z\in\cN\backslash\cN_\sD$, then $\{K\}\cup \hat{\o}_z
\subset \tilde{\o}_{K,z,qm}$
and $A_K\leq A_{K'} \; \forall K' \in \tilde{\o}_{K,z,qm}$;
\item if $z\in\cN_\sD$, then $K\in \tilde{\o}_{K,z,qm}$,
$\p\tilde{\omega}_{K,z,qm}\cap\Gamma_D \neq \emptyset$, and
$A_K\leq A_{K'} \; \forall K' \in \tilde{\o}_{K,z,qm}$.
\end{itemize}
The distribution of the coefficients $A_K$, $K\in\cT$, is said to be
locally {\em quasi-monotone} if it is quasi-monotone with respect to
every vertex $z\in\cN$.
\end{definition}
 For a function $v \in DG(\cT_h, s)$, we choose to employ the classical Lagrangian basis functions. 
 For a function $v \in DG^{0}(\cT_h, s)$,  since only  $\jump{v}|_{\cE \setminus \cE_{N}}$ is needed, from here to thereafter, we assume that $v$ satisfies the following conditions. 
 \begin{itemize}
 \item
The interior degrees of freedom of $v$ are zero, i.e., $v(\bx_{i,K}) =0, i=1, \cdots, m_{s}$.
 \item
Let $\{\bx_{i,F}\}_{i=1}^{s-1} (s \ge 1)$be the set of all Lagrange points on $F$ excluding the vertices $\bx_{s,F}$ and $\bx_{e,F}$ for the space $P_{k}(F)$. For each $F \in \cE_{I}$ and $\bx_{i,F} \in F$, let $v|_{K_{F}^{+}}(\bx_{i,F})=0$. 
\item
For each $z \in \cN$, we denote by $K_{z}$ be a element in $\hat w_{z}$ and let $v|_{K_{z}}(z) =0$.
\item When $F \in \cE_{N}$, we simply let $v|_{F} \equiv 0$.
 \end{itemize}
 It's not difficult to see that with the above four types of restrcitions, any function  $v \in DG^{0}(\cT_h, s)$ is uniquely determined. 

We first prove the following lemma.
\begin{lemma}
For any $v \in DG^{0}(\cT_h, s)$, there holds
\begin{equation}\label{qw-0}
	\|A^{1/2} \nabla v\|_{\cT_{h}} \le C  \sum_{F \in \cE \setminus \cE_{N}} h_{F}^{-1/2} A_{F}^{1/2} \|\jump{v}\|_{F}.
\end{equation}
where the constant $C$ is independent of the mesh and the jump of the coefficient $A(x)$,
\end{lemma}
\begin{proof}
For any $K \in \cT_{h}$, we first have that
\beq\label{qw-1}
\begin{split}
	\|\nabla v\|_{K} \lesssim  \sum_{z \in \cN_{K}} | v_{K}(z)| 
	+ \sum_{F \in \cE_{K}} \sum_{i=1}^{s-1} |v_{K}(\bx_{i,F})|.
\end{split}
\eeq
Fix any $z \in \cN_K$. If $K = K_{z}$, we have $v_{K}(z)=0$.
If $K \neq K_{z}$, applying the triangle equality and the property of $v$ yields
\beq
\begin{split}
 | v_{K}(z)|  \le \sum_{F \subset \tilde \o_{K,z,qm} \cap \cE_{z}} | \jump{v}|_F(z)|.
\end{split}
\eeq
Here we note that $\tilde \o_{K,z,qm}$ is open, so $\tilde \o_{K,z,qm} \cap \cE_{z}$ does not contain the two boundary edges $ \partial_{\tilde \o_{K,z,qm}} \cap \cE_{z}$.
Since $A(x)$ is monotone along $\tilde \o_{K,z,qm}$, we have
\beq
\begin{split}
 A_{K}^{1/2}| v_{K}(z)|  \le \sum_{F \subset \tilde \o_{K,z,qm} \cap \cE_{z}} A_{F}^{1/2}| \jump{v}|_F(z)|.
\end{split}
\eeq
Finally, for any $\bx_{i,F}$, $i=1, \cdots, s-1$ and $F \in \cE_{K}$ we have
\beq\label{qw-2}
A_{K}^{1/2}|v_{K}(\bx_{i,F})| = 
\begin{cases}
0 & \mbox{if }K = K_{F}^{+}\\
A_{K}^{1/2}| \jump{v}|_{F}(\bx_{i,F})|   & \mbox{if } K = K_{F}^{-}
\end{cases}
\le A_{F}^{1/2}| \jump{v}|_{F}(\bx_{i,F})|.
\eeq
Recall we assumed that $A_{F}^{-} \le A_{F}^{+}$. 
Combing \cref{qw-1}--\cref{qw-2}, we have for any $K \in \cT_{h}$,
\beq
\begin{split}
	\|A^{1/2}\nabla v\|_{K}& \lesssim  
	\sum_{z \in \cN_{K}}  \sum_{F \in \cE_{z}} 
	A_{F}^{1/2}| \jump{v}|_{F}(z)
	+  \sum_{F \in \cE_K}\sum_{i=1}^{s-1}A_{F}^{1/2}| \jump{v}|_{F}(\bx_{i,F})| 
	\\
	& \lesssim\sum_{z \in \cN_{K}}  \sum_{F \in \cE_{z}} A_{F}^{1/2} h_{F}^{-1/2} \| \jump{v}\|_{F}.
\end{split}
\eeq
and, hence, \cref{qw-0}. This completes the proof of the lemma.
\end{proof}


\begin{lemma}\label{lem:effi-cg-1}
Let $u_{s}^{\Delta} \in DG^{0}(\cT_h, s)$ be the solution of \cref{b-correction-A}. We have the following estimation for $\tri u_{s}^{\Delta}\tri$:
\beq \label{global-efficiency}
	\tri u_{s}^{\Delta}\tri \le C \| A^{1/2} \nabla (u - u_{k}^{cg})\|,
	\eeq
	where the constant $C$ is independent of the mesh and the jump of the coefficient of $A(x)$ and the norm $\tri \cdot \tri$ is defined in \cref{tri-norm}.
\end{lemma}
 \begin{proof}
 Frist, we observe that
 \begin{equation}\label{norm-bounds}
  \tri u_{s}^{\Delta} \tri = \sup_{v \in DG^{0}(\cT_h, s), \tri v\tri =1} {\mathcal{A}(u_{s}^{\Delta}, v)} =
  \sup_{v \in DG^{0}(\cT_h, s), \tri v\tri =1} r (v).
  \end{equation}
  Applying integration by parts,  the definition of $\tilde \bsigma_{k-1}^{cg}$ and \cref{jump-id},  we have
\begin{align}\label{effi:cg-a}
r(v) &
\!=\! (f - \nabla \cdot  \tilde \bsigma_{k-1}^{cg}, v)_{\O}
\!=\! (f,  v)_{\O} \!+\!  \sum_{K \in \cT_{h}} \left((\tilde \bsigma_{k-1}^{cg},  \nabla v)_{K} 
\!-\!
<\tilde \bsigma_{k-1}^{cg} \cdot \bn_{K},  v >_{\partial K} \right)
\notag\\
&= (f,  v)_{\O} - \sum_{K \in \cT_{h}}(A \nabla u_{k}^{cg},  \nabla v)_{K} 
+
\sum_{F \in \cE \setminus \cE_{N}}< \{ A \nabla u_{k}^{cg} \cdot \bn_{F}\}_{w}^{F} , \jump{ v}) >_{F}
\\
&= (f + \nabla \cdot A \nabla u_{k}^{cg}, v)_{\O}
-
\sum_{F \in \cE_{I}}< (\jump{ A \nabla u_{k}^{cg} \cdot \bn_{F}} ,  \{v\}_{F}^{w}) >_{F}.\notag
\end{align}

By the triangle and Cauchy Schwartz inequalities, and \cref{weight1}, we have for any $F \in \cE_{I}$,
\begin{equation}\label{bound-part2a}
\begin{split}
&< \jump{ A \nabla u_{k}^{cg} \cdot \bn_{F}} ,  \{v\}_{F}^{w}>_{F}
\lesssim
		\|  \jump{ A \nabla u_{k}^{cg} \cdot \bn_{F}}\|_{F}  
		(\omega^+\| v_{K_F^-}\|_{F}
		+\omega^-\|v_{K_F^+}\|_{F} )\\
&\lesssim
\|  A_{F,max}^{-1/2}\jump{ A \nabla u_{k}^{cg} \cdot \bn_{F}}\|_{F}  
		(\| A^{1/2}v_{K_F^-}\|_{F}
		+\|A^{1/2}v_{K_F^+}\|_{F} ).
\end{split}
\end{equation}
Now, applying the trace and Poincar\'e inequalities, we can bound for any $K \in \{K_{F}^{+}, K_{F}^{-}\}$,
\begin{equation}\label{bound-part2b}
\| A^{1/2}v|_{K}\|_{F} \lesssim h_{F}^{-1/2}\| A^{1/2}v\|_K  + h_{F}^{1/2}\| A^{1/2} \nabla v\|_K  \lesssim   h_{F}^{1/2}\| A^{1/2} \nabla v\|_K .
\end{equation}

Finally, applying \cref{norm-bounds}--\cref{bound-part2b} and the Cauchy Schwartz inequality, \cref{qw-0} and the fact that $\tri v \tri =1$, we have 
\begin{align}\label{effi:cg-aa}
&r(v)\le 
 \sqrt{\sum_{K \in \cT_{h}}h_{K}^{2}\|A^{-1/2}( f + \nabla  \cdot A \nabla u_{k}^{cg})\|_K^{2} }
 \sqrt{\sum_{K \in \cT_{h}}\| A^{1/2} \nabla v\|_K^{2}} \notag\\
 &+
\sqrt{\sum_{F \in \cE \setminus \cE_{N}} h_{F} A_{F,max}^{-1}\| \jump{ A \nabla u_{k}^{cg} \cdot \bn_{F}} \|_{F}^{2}}
 \sqrt{\sum_{K \in \cT_{h}}\| A^{1/2} \nabla v\|_K^{2}}
\notag\\
&\lesssim 
\sqrt{ {\sum_{K \in \cT_{h}}h_{K}^{2}\| f + \nabla  \cdot A \nabla u_{k}^{cg}\|_K^{2} }
 +
 {\sum_{F \in \cE \setminus \cE_{N}}  \dfrac{h_{F}}{A_{F,max}}\| \jump{ A \nabla u_{k}^{cg} \cdot \bn_{F}} \|_{F}^{2}}} \|A^{1/2} \nabla v\|_{\cT_{h}},
\notag \\
&\lesssim 
 \sqrt{\sum_{K \in \cT_{h}}h_{K}^{2}\| f + \nabla  \cdot A \nabla u_{k}^{cg}\|_K^{2} }
 +
 \sqrt{\sum_{F \in \cE \setminus \cE_{N}} \dfrac{h_{F}}{A_{F,max}}\| \jump{ A \nabla u_{k}^{cg} \cdot \bn_{F}} \|_{F}^{2}} \notag,
\end{align}
which, combining with the classical local efficiency results in \cref{efficiency-ele-residual} and \cref{efficiency--jump} yields \cref{global-efficiency}. This completes the proof of the lemma.
 \end{proof}

\begin{lemma}\label{lem:effi-cg-2}
Define the local error indicator
\begin{equation} \label{cg-local-indicator}
	 \eta_{\sigma,K} = \|A^{-1/2} \hat \bsigma_h^{cg} + A^{1/2} \nabla u_k^{cg} \|_K 
	\end{equation}
	where $ \hat\bsigma_{h}^{cg}$ is defined in \cref{lem:sigma-hat-1}.
	Then we have the following global efficiency bound:
	\begin{equation}\label{global-effi-cg}
	 \eta_{\sigma} \le
	 C  \| A^{1/2} \nabla (u - u_{k}^{cg})\|_{\cT_{h}} + \mbox{osc}(f),
\end{equation}
where the constant is independent of the mesh and the jump of the coefficient $A(x)$.
\end{lemma}
\begin{proof}
Similar to the proof of \cref{lem:effi-nc-1}, it is sufficient to prove
	\begin{equation}\label{global-effi-cg-a}
	\|  A^{-1/2}\bsigma_{s}^{\Delta} \|_{\cT_{h}} \le
	C   \| A^{1/2} \nabla (u - u_{k}^{cg})\|_{\cT_{h}} + \mbox{osc}(f).
\end{equation}

By the definition of $\bsigma_{s}^{\Delta}$,  \cref{dg-}, we have for each $K \in \cT_{h}$,
\beq
\begin{split}
	\|A^{-1/2} \bsigma_{s}^{\Delta}\|_{K} &\lesssim 
	 \sum_{F \in \cE_{K}} h_{F}^{1/2}\| A_{K}^{-1/2} \bsigma_{s}^{\Delta} \cdot \bn_{F}\|_{F} 
	\le  \sum_{F \in \cE_{K}} h_{F}^{1/2}\| A_{K}^{-1/2}\mathbf{S}({u_{s}^{\Delta}}) \cdot \bn_{F}\|_{F}\\
	&\le  \sum_{F \in \cE_{K}} h_{F}^{-1/2}\| A_{K}^{-1/2} A_{F} \jump{u_{s}^{\Delta}}\|_{F}
	\le  \sum_{F \in \cE_{K}} h_{F}^{-1/2}\| A_{F}^{1/2} \jump{u_{s}^{\Delta}}\|_{F}.
\end{split}
\eeq
This, combining with \cref{global-efficiency}, indicates that
\beq
\begin{split}
	\|A^{-1/2} \bsigma_{s}^{\Delta}\|_{\cT_{h}} & \lesssim \tri u_{s}^{\Delta}\tri \le C \| A^{1/2} \nabla (u - u_{k}^{cg})\|_{\cT_{h}}.
\end{split}
\eeq
This completes the proof of the lemma.

\end{proof}

	\subsection{Local Efficiency of Flux Recovery using POU for Conforming FEM}
	In this subsection, we aim to prove the local efficiency of the error indicator defined in 
\cref{nc-local-indicator}, where $\bsigma_{s}^{\Delta}$ is the solution to \cref{-POU-a}. 
For simplicity, we analyze the case for two dimensions and $s = 0$. The efficiency results can be extended to other $s$ values in two dimensions under a similar fashion.

To take a closer look of the local systems, for any $z \in \cN$, let $\triangle_z =\{K_{i,z}\}_{i=1}^{T_z}$ be the set of all elements, oriented counterclockwise, sharing $z$ as a common vertex. $T_{z}$ denotes the number of triangle elements in $\triangle_{z}$. Let $F_{i,z}$ be the edge in $\cE_{z}$ shared by $K_{i-1,z}$ and $K_{i,z}$ for $i=2, \cdots, T_z$. Let $F_{1,z} = \cE_{K_{1,z}} \cap \cE_{z} \setminus \{F_{2,z}\}$. When $z \in \partial \Omega$, let $F_{T_{z}+1,z} =\cE_{K_{T_{z},z}} \cap \cE_{z} \setminus \{F_{T_{z},z}\}$. When $z \in \cN_{I}$,  let $F_{T_{z}+1,z} = F_{1,z}$. See \cref{fig:sub1}-\cref{fig:sub2} for an illustration.

\begin{figure}[htbp]
    \centering
    \begin{subfigure}{0.35\textwidth}
        \centering
        \includegraphics[width=\linewidth]{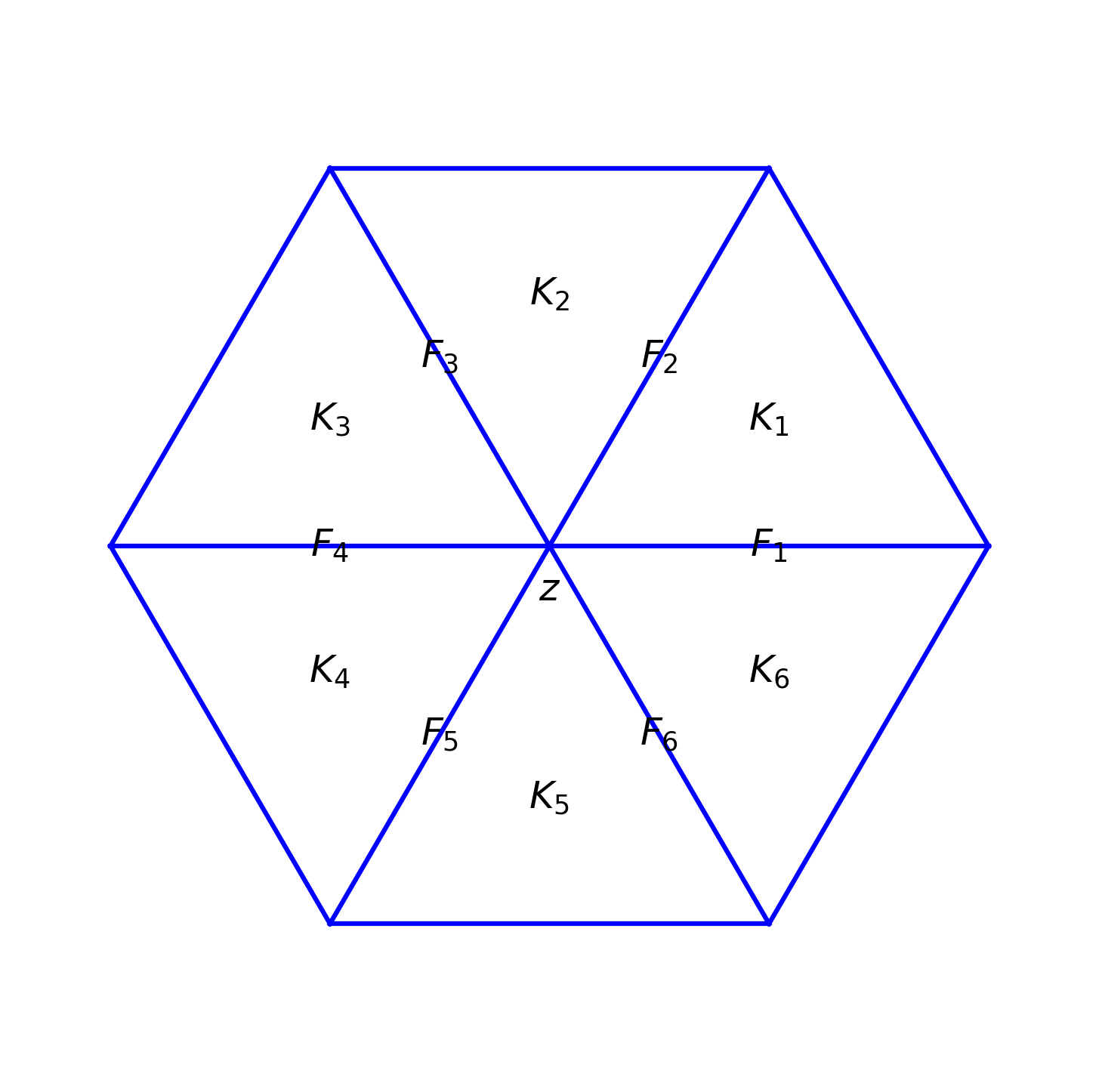}
        \caption{An interior vertex $z \in \cN_I$}
        \label{fig:sub1}
    \end{subfigure}
    \hfill
    \begin{subfigure}{0.35\textwidth}
        \centering
        \includegraphics[width=\linewidth]{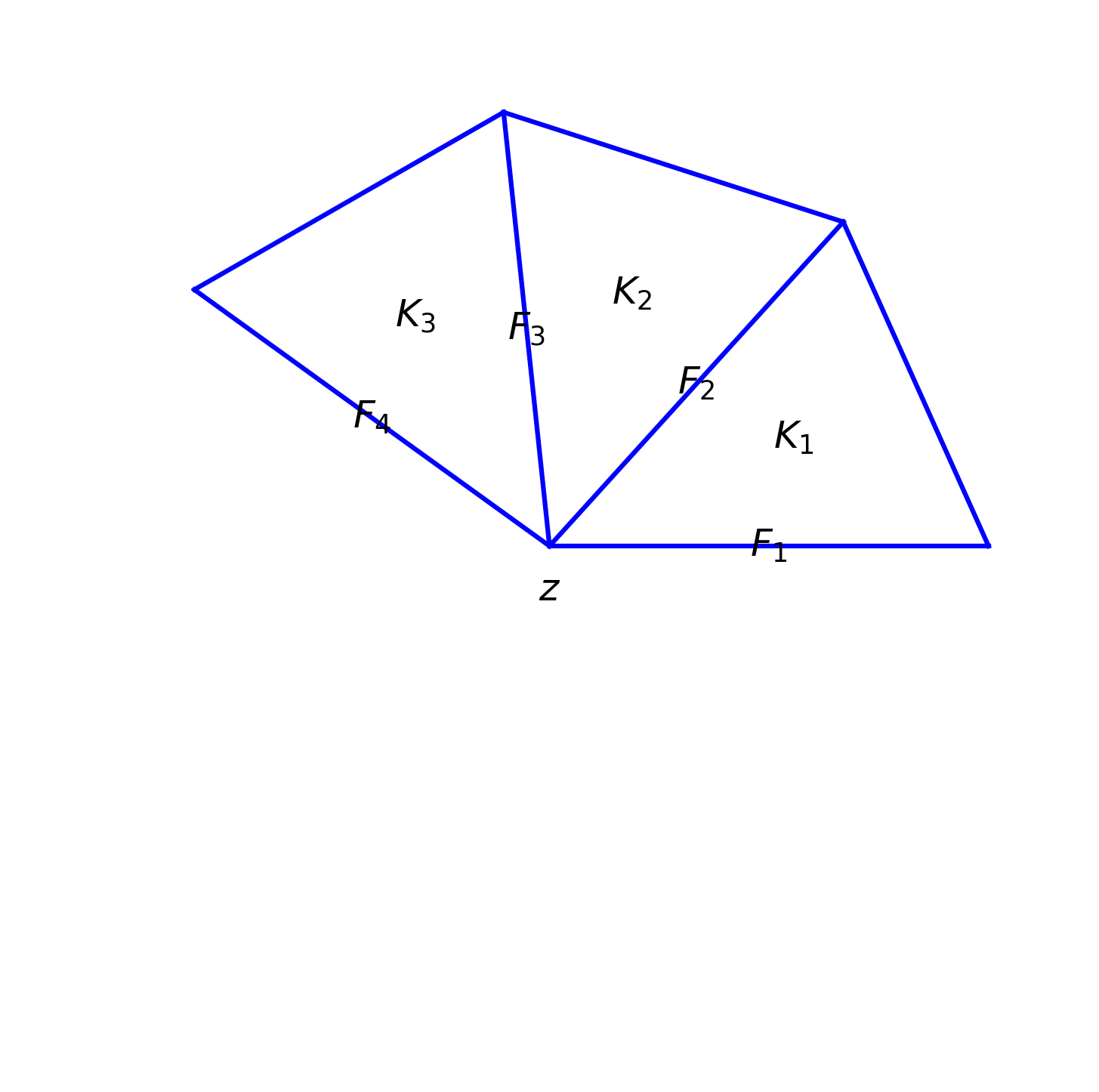}
        \caption{A boundary vertex $z \in \partial \O$}
        \label{fig:sub2}
    \end{subfigure}
    \label{fig:main} 
\end{figure}

Recall that  for any $z \in \cN$, we need to solve $\bsigma_z^{\Delta}$ such that
\[
\mathcal{A}( \bsigma_z^{\Delta}, v) =  r(v \lambda_{z}) \quad \forall v \in DG(\o_{z},0).
\]
The explicit formulas  for ${\bsigma_z^{\Delta}}$ are provided in \cref{appen-pouprogramming}.

\begin{lemma}\label{lem:monotone}
Assuming that for any $z \in \cN$, $\{K_{i,z}\}_{i=1}^{T_{z}}$ can be ordered in the way such that $\{A_{K_{i,z}}\}_{i=1}^{T_{z}}$ is monotone. Then for any $z \in \cN$ we have
\begin{equation}\label{estimate-x-z}
A_F^{-1/2} |x_{F,z}| \le \sum_{K \in \o_{z}} A_{K}^{-1/2} r_{K,z}, \forall F \in \cE_{z} \setminus \cE_{N}.
\end{equation}
The proof of this lemma is provided in \cref{appen:lemmproof}.
\end{lemma}
\begin{remark}
We note that the monotone assumption in \cref{lem:monotone} is slightly more restrictive than the quasi-monotone assumption.
\end{remark}


 \begin{lemma}\label{effi-cg-pou}
 The $\bsigma_{s}^{\Delta}$ defined in \cref{-POU-a} for $s=0$ in the two dimensions satisfies the following local efficiency result:
 \begin{equation}\label{local-efficiency-cg}
 \| A^{-1/2} \bsigma_{s}^{\Delta}\|_{K} \le C \|A^{1/2} \nabla (u - u_{k}^{cg})\|_{\o_{K}},
\end{equation}
 where the constant $C$ is independent of the mesh and the jump of the coefficient $A(x)$ and  $ u_{k}^{cg}$ is the conforming finite element solution to  \cref{CG-solution} for $k \ge 1$.
\end{lemma}
\begin{proof}
From \cref{-cg-null-correction}, we first observe that 
	\begin{equation}
	\begin{split}
		 \| A^{-1/2} \bsigma_{s}^{\Delta}\|_{K}^{2} \lesssim
		 \sum_{z \in \cN_{K} \cap \cN^{*}}  \| A^{-1/2}(\bsigma_{z}^{\Delta} -  \lambda_{z} \bsigma_z^{\#}))\|_{K}^{2}
		 + \sum_{z \in \cN_{K} \setminus \cN^{*}}  \| A^{-1/2}\bsigma_{z}^{\Delta}\|_{K}^{2}.
	\end{split}
	\end{equation}
	In the case when $z \in \cN^{*}$, applying the facts that
	$\|A^{-1/2}\nabla \bsigma_{z}^{\#}\|_{\o_{z}}=1$ and  $ \lambda_{z} ={(A^{-1}  \bsigma_z^{\Delta},  \bsigma_z^{\#})_{\o_{z}}}$,  and the Cauchy Schwartz inequality yields	
		\begin{equation}\label{null-correction-bound}
	\begin{split}
		& \| A^{-1/2}(\bsigma_{z}^{\Delta} -  \lambda_{z} \bsigma_z^{\#}))\|_{K}
		 \le \| A^{-1/2}\bsigma_{z}^{\Delta} \|_{K} +|\lambda_{z}|
		  \le 2\| A^{-1/2}  \bsigma_z^{\Delta} \|_{\o_{z}}.
	\end{split}
	\end{equation}
	It is, therefore sufficient to prove that 
	 \begin{equation}\label{local-efficiency-cg-a}
 \| A^{-1/2} \bsigma_{z}^{\Delta}\|_{K}\lesssim \|A^{1/2} \nabla e\|_{\o_{z}} \quad \forall z \in \cN_{K}.
\end{equation}
We prove the case when $z \in \cN_{I}$. For the other cases, the proof can be done similarly.
First, by the equivalence of norms, $A_{F} = \min (A_{F}^{+}, A_{F}^{-})$, and \cref{estimate-x-z} we have that
\begin{equation}\label{7-53}
\begin{split}
\| A^{-1/2} \bsigma_{z}^{\Delta} \|_K 
\lesssim \sum_{F \in \cE_{z} \cap \cE_{K}} \dfrac{h_F} { \sqrt{A_K}} | \bsigma_{z}^{\Delta}  \cdot \bn_{F}| \le 
 \sum_{F \in \cE_{z} \cap \cE_{K}} A_F^{-1/2} |x_{F,z}|
 \lesssim
  \sum_{K \in \o_{z}} A_{K}^{-1/2} r_{K,z}.
\end{split}
\end{equation}
From integration by parts, \cref{weight1},  \cref{jump-id}, and Cauchy Schwartz inequality, we have
for any $K \in \o_{z}$,
	\begin{align}\label{effi:cg-0}
& A_{K}^{-1/2}|r(\lambda_{z} 1_{K})|
=  A_{K}^{-1/2}\left|(f - \nabla \cdot  \tilde \bsigma_{k-1}^{cg}, \lambda_{z})_{K}\right|\\
&=  A_{K}^{-1/2}
\left|(f,  \lambda_{z})_{K} +  \left(\tilde \bsigma_{k-1}^{cg},  \nabla \lambda_{z})_{K} 
-
<\tilde \bsigma_{k-1}^{cg} \cdot \bn_{K},  \lambda_{z} 1_{K} >_{\partial K} \right) \right|
 \notag\\
&=  A_{K}^{-1/2}
\left|(f,  \lambda_{z})_{K} 
	-  ( A \nabla u_{k}^{cg},  \nabla \lambda_{z})_{K} 
+
\sum_{F \in \cE_{K} \cap \cE_{z}}<\{ A \nabla u_{k}^{cg} \cdot \bn_{F} \}_{w}^{F},  \jump{\lambda_{z} 1_{K}} >_{F} \right|
 \notag\\
&\le   \left|A_{K}^{-1/2}(f + \nabla \cdot A \nabla u_{k}^{cg},  \lambda_{z})_{K} \right|
+
\sum_{F \in\cE_{K} \cap \cE_{z}} A_{F,max}^{-1/2}\left|< (\jump{ A \nabla u_{k}^{cg} \cdot \bn_{F}} ,  
\jump{ \lambda_{z}1_{K}}) >_{F} \right|
 \notag\\
&\lesssim 
{A_{K}^{-1/2} h_{K}\| f + \nabla  \cdot A \nabla u_{k}^{cg}\|_K }
 +
{\sum_{F \in\cE_{K} \cap \cE_{z}} h_{F}^{1/2} A_{F,max}^{-1/2}\| \jump{ A \nabla u_{k}^{nc} \cdot \bn_{F}} \|_{F}}, \notag
\end{align}
where we have used the facts that
\[
	\| \lambda_{z}\|_K \lesssim h_{K} , \quad   \|\lambda_{z}\|_{F} \lesssim h_{F}^{1/2}.
\]
Finally, combing \cref{7-53}, \cref{effi:cg-0} with the classical local efficiency results in \cref{efficiency-ele-residual} and \cref{efficiency--jump} yields \cref{local-efficiency-cg-a}, and, hence, \cref{local-efficiency-cg}. This completes the proof of the lemma.
\end{proof}

\begin{remark}
From \cref{null-correction-bound}, we observe that adding the correction \( \bsigma_{z}^{\#} \) does not impact the lower efficiency bound. Since the upper reliability bound holds automatically, computing this term in practice is not strictly necessary. However, incorporating such corrections will likely result in a more accurate efficiency constant.

\end{remark}

\begin{lemma}\label{lem:effi-cg-4}
Define the local error indicator
\begin{equation} \label{cg-local-indicator-2}
	 \eta_{\sigma,K} = \|A^{-1/2} \hat \bsigma_h + A^{1/2} \nabla u_k^{cg} \|_K 
	\end{equation}
	where $ \hat\bsigma_{h}$ is defined in \cref{lem:sigma-hat-cg2}.
	Then we have the following local efficiency bound:
	\begin{equation}\label{global-effi-cg-2}
	 \eta_{\sigma,K} \le
	 C  \| A^{1/2} \nabla (u - u_{k}^{cg})\|_{\o_{K}} + \mbox{osc}(f),
\end{equation}
where the constant $C$ is independent of the mesh and the jump of the coefficient $A(x)$.
\end{lemma}
\begin{proof}
	\cref{global-effi-cg-2} is a direct consequence of the triangle inequality, \cref{bound-part1}, \cref{effi-cg-pou} and the efficiency result \cref{efficiency--jump}.
\end{proof}

\begin{remark}

In three dimensions, the null space of \cref{local-pro-z} has a dimension greater than \(1\). Moreover, since a counterclockwise ordering is not predefined, each vertex may require a separate local matrix.
\end{remark}

\bibliographystyle{siamplain}
\bibliography{references}

\appendix
\section{Programming the POU approach in 2D}\label{appen-pouprogramming}
We analyze the case for two dimensions and $s = 0$.
Recall that $\bsigma_z^{\Delta}$ satisfies that
	\begin{equation}
\mathcal{A}( \bsigma_z^{\Delta}, v) =  r(v \lambda_{z}) \quad \forall v \in DG(\o_{z},0).
\end{equation}
Test $v = 1_{K}$ for any $K \subset \o_{z}$ yields
 	\begin{equation}\label{local-equation}
\sum_{F \in \cE_{K} \cap \cE_{z} \setminus \cE_{N} } |F| \bsigma_{z}^{\Delta} \cdot \bn_{F} \mbox{sign}_{K}(F) = r(1_{K} \lambda_{z}).
\end{equation}
Define $|F| \bsigma_{z}^{\Delta} \cdot \bn_{F} \mbox{sign}_{z}(F) = x_{F,z}$ and $ r(1_{K} \lambda_{z}) := r_{K,z}$,
where $\mbox{sign}_{z}(F)$ is a sign function defined on $\cE_{z}$ such that
 $\mbox{sign}_{z}(F) =1 (-1)$ if $\bn_{F}$ is oriented clockwise (counterclockwise) with respect to $z$. 
 
%
%

 By definitions,  it's easy to verify that for all $K_{i,z}, i=1, \cdots, T_{z}$:
\begin{equation}\label{sign-property-a}
	\mbox{sign}_{z}(F_{i,z}) \mbox{sign}_{K_{i,z}}(F_{i,z})=1 \;
	\mbox{and} \;
	\mbox{sign}_{z}(F_{i+1,z}) \mbox{sign}_{K_{i,z}}(F_{i+1,z})=-1. 
\end{equation}

Applying \cref{sign-property-a}, we can rewrite \cref{local-equation} as
\beq\label{local-equatioT_z-1}
	x_{i,z}  - x_{i+1,z} = r_{i,z},  i=1, \cdots, T_{z},
\eeq
where $x_{i,z} = x_{F_{i},z}$ and $ r_{i,z} =  r_{K_{i},z}$.
 
 Define
 \[
M_{s} = \begin{bmatrix}
1 & -1 & 0 & 0 & \cdots & 0 \\
0 & 1 & -1 & 0 & \cdots & 0 \\
\vdots & \vdots & \vdots & \ddots & \ddots & \vdots \\
0 & 0 & 0 & \cdots & 1 & -1 \\
-1 & 0 & 0 & \cdots & 0 & 1 \\
\end{bmatrix}_{s \times s}
\tilde M_s\!\! = \begin{bmatrix}
1 & -1 & 0 & 0 & \cdots & 0 \\
0 & 1 & -1 & 0 & \cdots & 0 \\
\vdots & \vdots & \vdots & \ddots & \ddots & \vdots \\
0 & 0 & 0 & \cdots & 1 & -1 \\
a_{1} & a_{2} & a_{3} & \cdots & a_{s-1} & a_{s} \\
\end{bmatrix}_{{s \times s}}
\]
and
\[
\bx_{z} = \begin{pmatrix}
x_{1,z} \\
x_{2,z} \\
\vdots \\
x_{|\cE_{z}|-1,z} \\
x_{|\cE_{z}|,z} \\
\end{pmatrix},
\quad
\mathbf{r}_{z} = \begin{pmatrix}
r_{1,z} \\
r_{2,z} \\
\vdots \\
r_{|\cE_{z}|-1,z} \\
-\sum_{i=1}^{|\cE_{z}|-1}r_{i,z} \\
\end{pmatrix},
\quad
\tilde{\mathbf{r}}_{z} = 
\begin{pmatrix}
r_{1,z} \\
r_{2,z} \\
\vdots \\
r_{|\cE_{z}|-1,z} \\
0 \\
\end{pmatrix},
\]
where $|\cE_{z}|$ is the number of edges in $\cE_{z}$. Note that $|\cE_{z}| = T_{z}$ if $z$ is an interior vertex and $|\cE_{z}| = T_{z}+1$ if $z$ is a boundary vertex.
Note that the matrix $M_s$ does not depend on the mesh. Denote by $M_{s',s} \, (s' \leq s)$ the matrix sharing the first $s'$ rows of $M_s$. For the last row of $\tilde M_{s}$, we choose the values of $a_{i} >0$ depending on $A$ to ensure a unique solution and robustness with respect to $A$.
We analyze the the solution of \cref{local-equatioT_z-1} in three scenarios.

Case 1: $z$ is an interior vertex or  $z \subset \partial \O$ and $z$  shared by two Dirichlet boundaries. \cref{local-equatioT_z-1} is equivalent to solve
\begin{equation}
	M_{|\cE_{z}|} \bx_{z} =\mathbf{r}_{z}.
	\end{equation}
	Note that in the second case, where \( z \) is shared by two Dirichlet boundaries, testing \( K \in \cT_{z} \) provides only \( |\cE_{z}| - 1 \) equations. However, we can still use the same linear system, as the last equation in \( M_{|\cE_{z}|} \bx_{z} = \mathbf{r}_{z} \) is linearly dependent and does not affect the solution.

Since the above system has a one-dimensional kernel, we will solve for a specific solution that satisfies
\[
	 \bx_{z} =  \tilde M_{|\cE_{z}|}^{-1} \,\tilde{\mathbf{r}}_{z} .
\]
Note that $ \tilde M_{|\cE_{z}|}$ is non-singular and thus has a unique solution.
In particular, the solution has the following form:
\begin{equation}\label{x-solution}
	x_{1,z} = \sum_{i=1}^{\cE_z-1} \dfrac{ \bar \Lambda_{i,z}}{\Lambda_z} r_{i,z}, \quad
	x_{\cE_z,z} =- \sum_{i=1}^{\cE_z-1} \dfrac{\Lambda_{i,z}}{\Lambda_z} r_{i,z}, \quad
	x_{m,z} = \sum_{i=m}^{\cE_z-1} \dfrac{ \bar \Lambda_{i,z}}{\Lambda_z} r_{i,z}
	- \sum_{i=1}^{m-1} \dfrac{\Lambda_{i,z}}{\Lambda_z} r_{i,z},
	\end{equation}
where $\Lambda_{z} =\displaystyle \sum_{i=1}^{\cE_z} a_{i}$ and $\Lambda_{i,z} = \displaystyle\sum_{k=1}^{i} a_{i}$ for $i =1, \cdots, \cE_z -1$ and $\bar \Lambda_{i,z} = \Lambda - \Lambda_{i,z}$

Case 2. $z$ is a boundary vertex and $z$ is shared by two Neumann boundaries.  The vector $\bx_z$ can be uniquely obtained as
\[
 \bx_z[2: \cE_{z}-1] = B_z^{-1} \, \mathbf{r}_z[1:\cE_z-2],
\]
where $B_z$ is the matrix obtained by removing the first and last columns and the last two rows of $M_{\cE_{z} \times \cE_{z}}$. Note that we remove the first and last column since $\bsigma_{z}^{\Delta}=0$ on $\cE_{N}$, i.e.,  $\bx_z[1] =  \bx_z[\cE_{z}] =0$. 

Case 3. Finally, consider the case where $z$ is a boundary vertex, and $z$ is shared by one Neumann boundary and one Dirichlet boundary. The vector $\bx_z$ can be uniquely obtained as
\[
\bx_z[2: \cE_{z}] = C_z^{-1} \mathbf{r}_z, \quad \mbox{or} \quad
\bx_z[1: \cE_{z}-1] = C_z^{-1} \mathbf{r}_z,
\]
where $C_z$ is the matrix obtained by removing  the first or the last column of $M_{T_z \times (T_z+1)}$ depending on the location of the Riemann boundary. More precisely, the first column of $M_{T_z \times (T_z+1)}$ is removed if the Riemann boundary coincides with $F_{1,z}$. Otherwise, the last column of $M_{T_z \times (T_z+1)}$ is removed.

From here to thereafter, for each $z \in \cN$, we let
$a_{i} = A_{F_{i,z}}^{-1}$ for $i=1, \cdots, |\cE_{z}|$ in the last row of $\tilde M_s$

\begin{remark}
From above, we can see that the recovered flux can be explicitly defined in two dimensions.
The efficiency results for recovered flux in higher-order RT spaces can be proved similarly. 
\end{remark}

\section{Proof of \cref{lem:monotone}} \label{appen:lemmproof}
\begin{proof}
We prove the case when $z \in \cN_{I}$. The other cases can be proved similarly.
Without loss of generality, we assume that $\{A_{K_{i,z}}\}_{i=1}^{T_{z}}$ is increasing. Therefore, 
$\{A_{F_{i,z}}\}_{i=1}^{T_{z}}$ is also increasing.

From \cref{x-solution}, the triangle inequality, the fact that  $\{A_{K_{i,z}}\}_{i=1}^{T_{z}}$ is increasing, and the mean inequality, i.e., $ 2ab \le a^{2} + b^{2}$, we can bound $A_{F_{1},z}^{-1/2} |x_{1,z}| $ as follows:
\begin{equation}
\begin{split}
A_{F_{1},z}^{-1/2} |x_{1,z}| &\le A_{F_{1},z}^{-1/2} 
\left|\sum_{i=1}^{T_{z}-1} \dfrac{ \bar \Lambda_{i,z}}{\Lambda_z} r_{i,z}
 \right|=
 \left|\sum_{i=1}^{T_{z}-1} \dfrac{ \bar\Lambda_{i,z} A_{F_{1},z}^{-1/2} A_{K_{i,z}}^{1/2}}{\Lambda_z} A_{K_{i,z}}^{-1/2} r_{i,z}
 \right| \\
 &= 
 \left|\sum_{i=1}^{T_{z}-1}  \sum_{k=i+1}^{T_z-1}\dfrac{ A_{F_{k},z}^{-1}
 A_{F_{1},z}^{-1/2} A_{K_{i,z}}^{1/2}}{\Lambda_z} A_{K_{i,z}}^{-1/2} r_{i,z}
 \right| \\
  &\le
\sum_{i=1}^{T_{z}-1}  \sum_{k=i+1}^{T_z-1} \left|\dfrac{ A_{F_{k},z}^{-1/2}
 A_{F_{1},z}^{-1/2} }{\Lambda_z}  A_{K_{i,z}}^{-1/2}r_{i,z}
 \right|  \le
\sum_{i=1}^{T_{z}-1}  \sum_{k=i+1}^{T_z-1}   \left|\dfrac{ A_{F_{k},z}^{-1/2}
 A_{F_{1},z}^{-1/2} }{ A_{F_{k},z}^{-1} +
 A_{F_{1},z}^{-1}}  A_{K_{i,z}}^{-1/2}r_{i,z}  \right| \\
 & \le
  \sum_{i=1}^{T_{z}-1}  \sum_{k=i+1}^{T_z-1} \dfrac{1}{2} \left|A_{K_{i,z}}^{-1/2} r_{i,z}  \right| 
   \lesssim
\sum_{i=1}^{T_{z}-1}     \left|A_{K_{i,z}}^{-1/2} r_{i,z}  \right|.
\end{split}
\end{equation}
To estimate $A_{F_{T_{z}},z}^{-1/2} |x_{T_{z},z}|$, we use the fact $\{A_{F_{k},z}\}_{k=1}^{T_{z}}$ is increasing and $\dfrac{ \Lambda_{i,z}}{\Lambda_z} \le 1$
\begin{equation}
\begin{split}
A_{F_{T_{z}},z}^{-1/2} |x_{T_{z},z}| &\le A_{F_{T_{z}},z}^{-1/2} 
\left|\sum_{i=1}^{T_{z}-1} \dfrac{ \Lambda_{i,z}}{\Lambda_z} r_{i,z}
 \right|
 \le 
 \left|\sum_{i=1}^{T_{z}-1} A_{F_{T_{z}},z}^{-1/2} A_{K_{i,z}}^{1/2}  A_{K_{i,z}}^{-1/2}   r_{i,z}
 \right| \\
 & \le
  \sum_{i=1}^{T_{z}-1}   \left|A_{K_{i,z}}^{-1/2} r_{i,z}  \right|.
\end{split}
\end{equation}
Finally, to estimate $A_{F_{m},z}^{-1/2} |x_{m,z}|$ for $2 \le m \le T_{z}-1$, using similar arguments as above yields,
\begin{equation}
\begin{split}
A_{F_{m},z}^{-1/2} |x_{m,z}| &\le A_{F_{m},z}^{-1/2} 
\left|\sum_{i=m}^{T_{z}-1} \dfrac{ \bar \Lambda_{i,z}}{\Lambda_z} r_{i,z}
	- \sum_{i=1}^{m-1} \dfrac{\Lambda_{i,z}}{\Lambda_z} r_{i,z}  \right| \\
	 &\le 
 \left|\sum_{i=m}^{T_{z}-1}  \sum_{k=i+1}^{T_z-1}\dfrac{ A_{F_{k},z}^{-1}
 A_{F_{m},z}^{-1/2} A_{K_{i,z}}^{1/2}}{\Lambda_z} A_{K_{i,z}}^{-1/2} r_{i,z}
 \right| +
 \left|\sum_{i=1}^{m-1} A_{F_m}^{-1/2} A_{K_{i,z}}^{1/2}  A_{K_{i,z}}^{-1/2}   r_{i,z}
 \right|\\
 	 &\le 
 \left|\sum_{i=m}^{T_{z}-1}  \sum_{k=i+1}^{T_z-1}\dfrac{ A_{F_{k},z}^{-1/2}
 A_{F_{m},z}^{-1/2} }{A_{F_{k},z}^{-1}+ A_{F_{m},z}^{-1} } A_{K_{i,z}}^{-1/2} r_{i,z}
 \right| +
 \left|\sum_{i=1}^{m-1}   A_{K_{i,z}}^{-1/2}   r_{i,z}
 \right|\\
  	 &\le 
 \left|\sum_{i=m}^{T_{z}-1}  \sum_{k=i+1}^{T_z-1} A_{K_{i,z}}^{-1/2} r_{i,z}
 \right| +
 \left|\sum_{i=1}^{m-1}  A_{K_{i,z}}^{-1/2}   r_{i,z}
 \right|
   \lesssim
  \sum_{i=1}^{T_{z}-1}   \left|A_{K_{i,z}}^{-1/2} r_{i,z}  \right|.
\end{split}
\end{equation}
This completes the proof of the lemma.
\end{proof}

\end{document}